\documentclass[a4paper,11pt]{amsart}
\usepackage{version}
\usepackage{amssymb}
\usepackage{amsfonts}
\usepackage{graphicx}
\usepackage{tikz}
\usetikzlibrary{arrows}

\usepackage{color}
\usepackage{epstopdf}
\usepackage[english]{babel}
\usepackage{float}
\usepackage{cite}
\usepackage{tikz,pgf}
\usetikzlibrary{patterns,spy}

\usepackage{geometry}
\newtheorem{theorem}{Theorem}[section]

\newtheorem{lemma}[theorem]{Lemma}
\newtheorem{corollary}[theorem]{Corollary}
\newtheorem{proposition}[theorem]{Proposition}
\newtheorem{conjecture}[theorem]{Conjecture}

\theoremstyle{definition}
\newtheorem{definition}[theorem]{Definition}

\theoremstyle{remark}
\newtheorem{remark}[theorem]{Remark}

\numberwithin{equation}{section}

\renewcommand\bigskip{\medskip}

\def\to{\rightarrow}

\def\T{\mathbb T}

\def\N{\mathbb N}

\def\Q{\mathbb Q}
\def\R{\mathbb R}

\def\Z{\mathbb Z}
\def\E{\mathbb E}

\begin{document}

\title[Furstenberg's intersection conjecture]
{A proof of Furstenberg's conjecture on the intersections of $\times p$ and $\times q$-invariant sets
}

\author[Meng Wu]{Meng Wu}
\address{Department of Mathematical Sciences, P.O. Box 3000, 90014 
University of Oulu, Finland}
\email{meng.wu@oulu.fi}

\thanks{We acknowledge the postdoc fellowships supported by Academy of Finland (Centre of Excellence in Analysis and Dynamics Research) and  ERC grant 306496.}


\begin{abstract}
We prove the following  conjecture of Furstenberg (1969): if $A,B\subset [0,1]$ are closed and invariant under $\times p \mod 1$ and  $\times q \mod 1$, respectively, and if $\log p/\log q\notin \Q$, then for all real numbers $u$ and $v$,
$$\dim_{\rm H}(uA+v)\cap B\le \max\{0,\dim_{\rm H}A+\dim_{\rm H}B-1\}.$$
We obtain this result as a consequence of our study on the intersections of incommensurable self-similar sets on $\R$.
Our methods also allow us to give upper bounds for dimensions of arbitrary slices of planar self-similar sets satisfying SSC and certain natural irreducible conditions.
\end{abstract}

\maketitle


\section{Introduction}\label{intro}

\subsection{Background and history}

This paper is concerned with Furstenberg's problem \cite{Furstenberg69} about the intersections of Cantor sets. The Cantor sets under consideration are dynamically defined, that is,  they are either invariant sets or attractors of certain dynamical systems. Let $(X,f)$ be a dynamical system  where  $f: X\to X$ is a measurable map on a compact metric space  $X$. Many important dynamical properties of $f$ are displayed by its invariant sets. 
Supposing that we are given two dynamical systems $(X,f)$ and $(X,g)$,  it is reasonable to expect that information about common dynamical features of  $f$ and $g$ can be obtained by comparing their respectively invariant sets.  We are particularly interested in systems $(X,f)$ and $(X,g)$ which are  arisen from two arithmetically or geometrically ``independent'' origins. In this case,  one expects that the two systems should share as few common structures as possible and thus an $f$-invariant set should intersect a $g$-invariant set in as small a set as possible.


Furstenberg has given in \cite{Furstenberg69} some quantitative formulations of the above philosophy. Let $\dim$ denote a dimension function for subsets of $X$ (e.g. Hausdorff dimension). Following Furstenberg, we say that $f$ and $g$ are {\em transverse} if  
$$\dim A\cap B\le \max \{0, \dim A+\dim B-\dim X\}$$
for all  closed sets $A$  and $B$ which are $f$- and  $g$-invariant, respectively.
The present work was motivated by a conjecture of Furstenberg concerning the transversality of two arithmetically ``independent'' systems. 

Two positive real numbers   $a$ and $b$ are said to be {\em multiplicatively independent},  denoted by $a\nsim b$,  if $\log a/\log b\notin \Q.$  For a natural number $m\ge2$, let $T_m: x\mapsto mx \mod 1$ be the $m$-fold map of the unit interval.
We use $\dim_{\rm H}A$ to denote the Hausdorff dimension of a set $A$. 
Furstenberg conjectured that two dynamics $T_p$ and $T_q$ with $p\nsim q$ are transverse. More precisely,
\begin{conjecture}[Furstenberg, \cite{Furstenberg69}]\label{Intersection conjecture}
Assume that $p\nsim q.$ 
Let $A_p, B_q \subset[0,1]$ be closed sets which are invariant under $T_p$ and $T_q$, respectively. Then for all real numbers $u$ and $v$, 
$$\dim_{\rm H} (uA_p+v)\cap B_q\le \max \{0,\dim_{\rm H}A_p+\dim_{\rm H}B_q-1\}.$$
\end{conjecture}

In this paper, we prove Conjecture \ref{Intersection conjecture}. 
We point out that Conjecture \ref{Intersection conjecture} is closely related to another conjecture of Furstenberg about expansions of real numbers in different bases, which is stronger  and remains open. For $x\in [0,1]$, we denote the orbit of $x$ under the map $T_m$ by $\mathcal{O}_m(x)=\{T_m^k(x): k\in \N\}$. 
\begin{conjecture}[Furstenberg, \cite{Furstenberg69}]\label{orbit conjecture}
If $p\nsim q$, then for each $x\in [0,1]\setminus \Q$, we have
\begin{equation}\label{eq conj orb 1}
\dim_{\rm H} \overline{\mathcal{O}_p(x)}+\dim_{\rm H} \overline{\mathcal{O}_q(x)}\ge 1.
\end{equation}
\end{conjecture}

Suppose that $p\nsim q$, $A_p$ is a closed $T_p$-invariant set and  $B_q$ is a closed $T_q$-invariant set, and $\dim_{\rm H}A_p+\dim_{\rm H}B_q<1$. Then Conjecture \ref{Intersection conjecture} implies that $\dim_{\rm H}A_p\cap B_q=0$, while  Conjecture \ref{orbit conjecture} predicts that 
$A_p\cap B_q\subset \Q$. In this respect, Conjecture \ref{orbit conjecture} is much stronger than Conjecture \ref{Intersection conjecture}.  It seems that Conjecture \ref{orbit conjecture}  is still far out of reach of current methods. Nevertheless, {as observed already by Furstenberg}, using Conjecture \ref{Intersection conjecture} one can obtain some partial results towards Conjecture \ref{orbit conjecture}: the set of $x\in [0,1]$ which do not satisfy \eqref{eq conj orb 1} has Hausdorff dimension zero.   See Theorem \ref{thm zero dim exception set for orbit conjecture} for a detailed proof. 

The aforementioned  conjectures belong to the broad category of rigidity problems about $\times p$ and $\times q$ dynamics, where there is a rich literature (see e.g. the survey paper of Lindenstrauss \cite{Lindenstrauss2005} and the references therein). The study of rigidity properties between $\times p$ and $\times q$ dynamics (when $p\nsim q$) was initiated by Furstenberg in his landmark paper \cite{Furstenberg67}. In that paper, Furstenberg established the celebrated Diophantine result: if  $p\nsim q$, then the unit interval itself is the only (infinite) closed set which is both $T_p$ and $T_q$ invariant. He has famously conjectured that the measure version of this should be also true: any Borel probability measure on the unit interval invariant under $T_p$ and $T_q$ is a linear combination of Lebesgue measure and an atomic measure supported on finitely many rational points. The best partial result towards this conjecture is due to  Rudolph and Johnson \cite{Rudolph90, Johnson92} who proved the conjecture under the assumption of positive entropy. The research along this line has been fruitful and influential, and it has led to deep advances in Diophantine approximation and homogeneous dynamics (see \cite{Lindenstrauss2005}).

In another direction, Conjecture \ref{Intersection conjecture} can  also be regarded as a problem about slices of fractal sets. Note that the set $(uA_p+v)\cap B_q$ is, up to an affine coordinate change, the intersection of the product set $A_p\times B_q$ with the line ${\ell}_{u,v}=\{(x,y):y=ux+v\}$. By a classical result of Marstrand \cite{Marstrand54}, for any Borel set $E\subset \R^2$ and each $u\in \R$, Lebesgue almost every $v\in \R$ satisfies 
$$\dim_{\rm H} E\cap {\ell}_{u,v}\le \max\{0,\dim_{\rm H}E-1\}.$$
In general, this is only an almost every result, and there could be exceptional pairs $(u,v)$ for which the above inequality fails. In most cases,  the set of exceptional $(u,v)$ is quite difficult to analyze. 

While explicitly determining the exceptional set is in general intractable, for certain fractal sets with regular arithmetical or geometrical structures, it is widely believed that  the exceptional set should be very small and could  only be caused by some evident algebraic or combinatorial reasons.  For $A_p, B_q$ as in Conjecture \ref{Intersection conjecture}, the set $A_p\times B_q$ is such an example, for which it is clear that certain lines parallel to the axes are exceptional for the slice result, and Conjecture \ref{Intersection conjecture} predicts that  these lines are the only exceptions. 

There is a rich literature about  generic slices of various fractal sets, see e.g.  \cite{Marstrand54,Hawkes75, KP1991, MattilaBook1,  FJ1999, FM1996, BFS2012,SS2015}. However, very little is known about specific slices, and there were few partial results concerning Conjecture \ref{Intersection conjecture} before the present paper. 
The first and perhaps also the best one is due to Furstenberg \cite[Theorem 4]{Furstenberg69}. His result states that under the assumption of the conjecture, if  $\overline{\dim}_{\rm B}(u_0A+v_0)\cap B=\gamma>0$ for some $u_0\neq 0,v_0\in \R$, then for Lebesgue almost $u\in \R$ there is $v$ such that $\dim_{\rm H}(uA_p+v)\cap B_q\ge\gamma$. From the last assertion, it is not hard to deduce that  in this case, we must have $\dim_{\rm H}A_p+\dim_{\rm H}B_q>1/2$ (see \cite[Theorem 7.9]{Hochman2014} for the deduction). Thus, under the assumption $\dim_{\rm H} A_p+\dim_{\rm H}B_q\le 1/2$, Furstenberg's result confirms Conjecture \ref{Intersection conjecture}. We will return back to \cite[Theorem 4]{Furstenberg69} in Subsection \ref{subsection construction of CP-dist 1}.  We would like to mention that the  technique (namely, CP-process) Furstenberg introduced and used in  \cite{Furstenberg69} is also important for the present work,   it will be one of the main ingredients for our proof of Conjecture \ref{Intersection conjecture}.


Recently, Feng,  Huang and Rao \cite{FHR} studied affine embeddings between incommensurable self-similar sets and, as a consequence, they showed that if $p\nsim q$, then for $T_p$-invariant self-similar set $E$ and $T_q$-invariant self-similar set $F$, there exists a (non-effective) positive constant $\delta$ depending on $E$ and $F$ such that the Hausdorff dimension of  the intersection of $F$ with each $C^1$-{diffeomorphic} image of $E$ does not exceed $\min\{\dim_{\rm H}E,\dim_{\rm H}F\}-\delta$. 
Later, Feng \cite{FengPersonal} obtained some effective versions of the results of  \cite{FHR}, but  these effective versions are still far from  sufficient for proving Conjecture \ref{Intersection conjecture}. Feng \cite{FengPersonal} also constructed, for any $s,t\in (0,1)$ and $\epsilon >0$,  a $T_p$-invariant set $A$ of dimension $s$ and a $T_q$-invariant set $B$ of dimension $t$ which verify Conjecture \ref{Intersection conjecture} with a loss of $\epsilon$.

Finally, we note that the slice problem may be considered as  ``dual" to the projection problem for fractal sets. In that direction, there is a dual version of Conjecture \ref{Intersection conjecture}, also due  to Furstenberg and recently settled by Hochman and Shmerkin \cite{HS2012}  (some special cases by Peres and Shmerkin\cite{PS2009}), which asserts that under the assumptions of Conjecture \ref{Intersection conjecture},  for each orthogonal projection $P_\theta$ from  $\R^2$ to $\R$ with direction $\theta$ not parallel to the axes,  we have
$$\dim_{\rm H} P_\theta(A_p\times B_q)=\min\{1,\dim_{\rm H}(A_p\times B_q)\}.$$
Recently, there has been considerable
 interest in the study of  projections of dynamically defined Cantor sets, see for instance the survey paper of Shmerkin \cite{Shmerkin2015Survey} and the references therein for more details.

\subsection{Statements of general results}

We prove a more general statement about intersections of regular homogeneous self-similar sets on $\R$ (see below for the definition) under natural irreducibility assumptions. Conjecture \ref{Intersection conjecture} will be a consequence of this general result.

We first recall some relevant definitions. An {\em iterated function system} (IFS) on $\R^d$ is a finite family $\{f_i\}_{i=1}^m$ of strictly contracting maps $f_i:\R^d\to \R^d$.  Its {\em attractor}  is the unique non-empty compact set $X\subset \R^d$ satisfying
$$X=\bigcup_{i=1}^mf_i(X).$$
The IFS $\{f_i\}_{i=1}^m$ is called {\em self-similar} if each map $f_i$ is a similarity transformation. In this case, the attractor  is called a self-similar set.

A self-similar IFS $\{f_i\}_{i=1}^m$ defined on the line $\R$ is said to be {\em regular} and {\em $\lambda$-self-similar} if it satisfies the following conditions:
\begin{itemize}
\item[(1)] {\em  Regular condition:} there exists an open interval $J$ such that $f_i(J)\subset J$ for each $i$ and $f_i(J) \cap f_j(J)=\emptyset$ for $i\neq j$;
\item[(2)] {\em   $\lambda$-self-similar condition:} there exists $0<\lambda<1$ such that each $f_i$ is of the form $f_i(x)=\lambda x+t_i$.
\end{itemize}
 The attractor of a regular and $\lambda$-self-similar IFS   will be called  a {\em regular $\lambda$-self-similar set}.

We use $\overline{\dim}_{\rm B}$ to denote upper box-counting dimension.
\begin{theorem}\label{main theorem 1}
Assume that $\alpha,\beta\in (0,1)$ with $\alpha\nsim \beta$.
Let $C_\alpha\subset \R$ be a regular $\alpha$-self-similar set and let $C_\beta\subset \R$ be a regular $\beta$-self-similar set. Then for all real numbers $u$ and $v$, we have 
$$\overline{\dim}_{\rm B} (uC_\alpha+v)\cap C_\beta\le \max\{0,\dim_{\rm H}C_\alpha+\dim_{\rm H}C_\beta-1\}.$$
\end{theorem}

If we compare Theorem \ref{main theorem 1} and Conjecture \ref{Intersection conjecture}, we notice that in Theorem \ref{main theorem 1}, $\alpha,\beta$ are real numbers, and moreover we consider the upper box-counting dimension of intersections.

From Theorem \ref{main theorem 1},  we can deduce  a {slightly} stronger result than what is stated in Conjecture \ref{Intersection conjecture}.
\begin{theorem}\label{theorem stronger version of intersection conj}
Under the assumptions of Conjecture \ref{Intersection conjecture}, we have for all real numbers $u$ and $v$,  $$\overline{\dim}_{\rm B} (uA_p+v)\cap B_q\le \max \{0,\dim_{\rm H}A_p+\dim_{\rm H}B_q-1\}.$$  
\end{theorem}

\smallskip

\begin{remark}\label{Remark to thm main 1}
(1) One deduces Theorem \ref{theorem stronger version of intersection conj} from Theorem \ref{main theorem 1} by using the fact that if  $A\subset [0,1]$ is a  closed $T_m$-invariant set,  then for any $\epsilon >0$, there exists $k\in \N$ and a  regular $1/m^k$-self-similar set $\widetilde{A}$ such that $A\subset \widetilde{A}$ and $\dim_{\rm H}A\ge \dim_{\rm H}\widetilde{A}-\epsilon$. See Section \ref{section Proofs of the remaining results} for  the detailed proof.

(2) In Theorem \ref{main theorem 1}, we only consider regular  $\lambda$-self-similar IFSs, but it also works for some other cases. For example, the same proof works if the regular condition is replaced by the {\em strong separation condition} (SSC). 

(3) Our approach is purely ergodic theoretical, it is quite flexible and can be extended to more general settings.  A natural generalization of Theorem \ref{main theorem 1} is to consider intersections of linear and non-linear IFS attractors. Under certain natural circumstances, one should expect similar dimension bounds as above for the intersections.  We expect that our methods could be developed further to treat these problems. 

(4) Theorem \ref{main theorem 1} has consequences on problems of embeddings between self-similar sets as studied in \cite{FHR}. See Section \ref{section Proofs of the remaining results}  for details.

\end{remark}

Our next result concerns slices of self-similar sets on the plane with irrational rotation. 

Let $\{f_i\}_{i=1}^m$ be a homogeneous self-similar IFS on $\R^2$, where for fixed $\lambda\in (0,1)$ and $\xi\in [0,1)$,  each $f_i:\R^2\to\R^2$ is defined by
$$f_i(x)=\lambda O_\xi x+t_i, $$  
with $t_i\in \R^2$ and $O_\xi$ being the rotation matrix of angle $2\pi\xi\in [0,2\pi)$.

\begin{theorem}\label{main theorem 2}
Let $X$ be a self-similar set corresponding to an  IFS as above. Suppose that $\xi$ is irrational and the IFS $\{f_i\}_i$ satisfies the strong separation condition. Then 
$$\overline{\dim}_{\rm B}(X\cap {\ell})\le \max\{0,\dim_{\rm H} X-1\}$$
for any line ${\ell}$ of $\R^2$.
\end{theorem}

\begin{remark}
 The irrationality condition for $\xi$ is necessary, as we can see from the 4-corner 1/3-Cantor set (i.e., the product of the classical 1/3-Cantor set $C$ with itself): certain lines parallel to the $x$ or $y$-axes intersect $C\times C$ in a set which is a copy of $C$. 

\end{remark}

We note that Theorems  \ref{theorem stronger version of intersection conj} and \ref{main theorem 2} have been simultaneously and independently proved by P. Shmerkin \cite{Shmerkin2016} using completely different (additive combinatorial) methods.

\subsection{Strategy of the proof}
Let us briefly describe our strategy for proving Theorem \ref{main theorem 2}. The proof of Theorem \ref{main theorem 1} follows the same strategy, but is a bit more technical. For a set $A\subset \R^2$, we denote by $N_{\delta}(A)$  the minimal number of  balls of diameter $\delta$ needed to cover $A$.

Let $X$ be a self-similar set satisfying the conditions of Theorem \ref{main theorem 2}.  Our overall strategy is to show that whenever there exists a line $\ell_0$ such that $\overline{\dim}_{\rm B}X\cap \ell_0=:\gamma>0$, then we must have $\dim_{\rm H}X\ge 1+\gamma$. To prove this,  we proceed to show that for any $\epsilon>0$ and all large enough $n$ there exist  $E_n^\epsilon\subset X$ and  a set of angles $F_n^\epsilon\subset [0,2\pi)$ satisfying the following properties (1)-(3). 
\begin{itemize}
\item[(1)] $N_{2^{-n}}(E_n^\epsilon)\le 2^{n\epsilon}$; 
\item[(2)] $N_{2^{-n}}(F_n^\epsilon)\ge 2^{n(1-\epsilon)}$; 
\item[(3)] For each $t\in F_n^\epsilon$ there exists a line $\ell_t$ with angle $t$ intersecting $E_n^\epsilon$ such that  $\inf_{x\in X}N_{2^{-n}}\left((X\cap \ell_t)\setminus B(x,r_0)\right)\ge 2^{n(\gamma-\epsilon)}$ where $r_0=r(\epsilon)>0$ is some constant not depending on $n$. 
\end{itemize}
From these estimates, one can deduce that $\overline{\dim}_{\rm B} X\ge 1+\gamma$. Since the self-similar set $X$ has equal Hausdorff and upper box dimensions, we get $\dim_{\rm H} X\ge 1+\gamma$; {see Section \ref{section Proof of main Theorem 1}.}

To show the existence of the sets  $E_n^\epsilon$ and $F_n^\epsilon$  described above, we use ergodic methods. We consider the dynamical system $(X, W)$ where $W$ is the inverse map of the IFS $\{f_i\}_{i=1}^m$ on $X$, {that is, the restriction of $W$ on $f_i(X)$ is $f_i^{-1}$.} Then $W$ is expanding and rotating, for each $k\ge 1$ the map  $W^k$ transforms a slice $\ell\cap X$ into a finite family $L_k(\ell)$ of slices and the angle of each slice in $L_k(\ell)$ is rotated by $-k\xi$ compared to that of $\ell$. For $z\in \ell\cap X$, we denote by $S(\ell,z,k)$ the unique slice in $L_k(\ell)$ containing $W^k(z)$. 

Now, {for any $\epsilon>0$,} we would like to find a slice $\ell\cap X$ and a  point $z\in \ell\cap X$ such that there exists a set $E_n^\epsilon\subset X$ satisfying the following:
\begin{itemize}
\item[(i)] $N_{2^{-n}}(E_n^\epsilon)\le 2^{n\epsilon}$; 
\item[(ii)] The set $F_n^\epsilon(z):=\{-k\xi\mod 2\pi: W^k(z)\in E_n^\epsilon\}$ satisfies $N_{2^{-n}}(F_n^\epsilon(z))\ge 2^{n(1-\epsilon)}$;
\item[(iii)] For {MOST} $k\in \{i\in\N: W^i(z)\in E_n^\epsilon\}$, we have 
\begin{equation}\label{eq intro strategy 1}
\inf_{x\in X}N_{2^{-n}}\left(S(\ell,z,k)\setminus B(x,r_0)\right)\ge 2^{n(\gamma-\epsilon)},
\end{equation}
{where MOST means such $k$'s have relative density $1-\epsilon$ in $\{i\in\N: W^i(z)\in E_n^\epsilon\}$.}
\end{itemize}
To achieve this goal, we first construct an ergodic $W$-invariant measure $\nu$ with positive entropy $h(\nu,W)>0$ such that for $\nu$-a.e. $z$, there exists some ``good" slice $\ell\cap X$ such that $z\in \ell\cap X$ and the estimate \eqref{eq intro strategy 1} holds for most $k\in \N$. Such a measure $\nu$ will be constructed in two steps. First, based on the initial slice $\ell_0\cap X$ with upper box dimension $\gamma$, we apply  Furstenberg's CP-process machinery to create a rich family of 
``nice" measures $\mu$ which are supported on slices of $X$, where  ``nice" roughly means that for $\mu$-a.e. $z$ on the supporting slice $\ell\cap X$ of $\mu$, \eqref{eq intro strategy 1} holds for most $k\in \N$.  Then a beautiful argument due to Hochman and Shmerkin \cite[Theorem 2.1]{HS2015}, which relates the small-scale structure of a measure to the distribution of $W$-orbits of its almost every point, will enable us to construct a $W$-invariant measure $\nu$ based on a  ``nice" measure provided by Furstenberg's CP-process. We show that this $W$-invariant measure $\nu$ admits the desired properties.

After having constructed such a $W$-invariant (ergodic) measure $\nu$, we apply our third ingredient, which is a general result in ergodic theory and  a consequence of Sinai's factor theorem, to show that the space $X$ can be partitioned (up to a part of small $\nu$-measure) into finitely many subsets $\cup_j A_j$ such that for $\nu$-a.e. $z$ and for each $j$ the set $E_n^\epsilon:=A_j$  satisfies the above properties (i) and (ii). 

We would like to mention that if we could prove that  the measure $\nu$ is weak-mixing (or more precisely, the spectrum of the system $(X,W,\nu)$ does not contain $\xi$), then it is easy to show that for any measurable set $A\subset X$ with $N_{2^{-n}}(A)\le 2^{n\epsilon }$  and $\nu(A)>0$, the set $E^\epsilon_n:=A$ satisfies the required  properties (i) and (ii) for $\nu$-a.e. $z$.  But from the construction of $\nu$, it seems difficult to get any information about the mixing or spectral properties of $\nu$. Instead, we have Sinai's factor theorem at our disposal, which provides us a Bernoulli factor system of $(X,W,\nu)$ with the same entropy as that of $\nu$, so we can first establish the required properties in the factor system and then ``transfer" the results back to the original system $(X,W,\nu)$. 
We note that the application of Sinai's factor theorem in the study of the kind of problems  considered in the present paper seems  new  and we hope that it may be useful for investigating other related questions.

For proving Theorem \ref{main theorem 1}, we follow in principle the same scheme as described above, but instead of considering a single transformation on $K=C_\alpha\times C_\beta$, we consider a skew product $U$ on $K\times [0,1)$.  The {component} of the map $U$ on $K$ is induced by the inverse maps of the defining IFSs of $C_\alpha$ and $C_\beta$ and  has the effect that it transforms a slice $\ell\cap K$ into finitely many pieces of slices whose slopes are changed in a way similar as the irrational rotation of angle $\theta=\log \alpha/\log \beta$ comparing to that of $\ell\cap K$.

There will be three main steps in the proof of Theorem \ref{main theorem 1}, as for the case of Theorem \ref{main theorem 2}.
First, assuming the existence of a  slice $\ell_0\cap X$ with upper box dimension $\gamma>0$, we construct a CP-distribution which is supported on ``nice" slice measures (with dimension $\gamma$) on $K$. Then based on these  ``nice" measures, we construct a $U$-invariant (ergodic) measure $\nu_\infty$ whose marginal on $K$ satisfies some similar ``nice slice" properties as that of $\nu$ (i.e., almost every point with respect to the marginal of $\nu_\infty$ lies on a ``good" slice of $K$). After the construction of such a measure $ \nu_\infty$, we proceed to the last step: apply our ergodic theoretic result to the system $(K\times [0,1),U,\nu_\infty)$ and conclude the proof.

\subsection{Organization of the paper}

In Section \ref{section notation}  we present some general notation, and collect some notions and basic properties of symbolic spaces, entropy, dimension and dynamical systems. In Section \ref{section CP-process}  we recall the CP-process theory. 
Sections \ref{section Constructions of CP-distributions on K}-\ref{section Proof of Theorem main 1} are devoted to the proof of Theorem \ref{main theorem 1}.
In Section  \ref{section Constructions of CP-distributions on K} we construct an ergodic CP-distribution which is supported on slice measures of $C_\alpha\times C_\beta$. In Section  \ref{subsection special measure 1} we define the skew-product $U$  and  construct the $U$-invariant measure $\nu_\infty$. In Section \ref{section ergodic results} we state and prove our general ergodic theoretic result. In Section \ref{section Proof of Theorem main 1} we prove Theorem \ref{main theorem 1}. In Section \ref{section Proof of Theorem main 2} we sketch the proof of Theorem \ref{main theorem 2}. In Section 
\ref{section Proofs of the remaining results} we present
an application of Theorem \ref{main theorem 1} on embeddings of self-similar sets, and we complete proofs of the remaining statements.

\subsection{Acknowledgements}

We gratefully acknowledge the helpful suggestions of P. Shmerkin on the presentation of the material in Section 4.
We also wish to express our thanks to A. Algom, A.H. Fan, M. Hochman, E. J\"arvenp\"a\"a,  M. J\"arvenp\"a\"a, T. Orponen  and V. Suomala for their useful comments on the writing of the paper. We are particularly indebted to the anonymous referee for a very through reading and many helpful suggestions which greatly improved the presentation of the paper.

\subsection{Summary of notation} For the reader's convenience, we summarize our main notation conventions in the following table.

\medskip

\begin{tabular}{l r @{\ \ \ \ } l}

\hline

$B(x,r)$ & &The closed ball of radius $r$ around $x$. \\

\smallskip

$\mathcal{P}(X)$ & &Space of probability measures on $X$. \\

\smallskip

$\mu, \nu, \eta, \upsilon,\vartheta$ & &Measures. \\

\smallskip

$P,Q$ & &Probability distributions (elements of $\mathcal{P}(\mathcal{P}(X))$). \\

\smallskip

$\mu|_A$ & &Restriction of $\mu$ on $A$. \\

\smallskip

$\underline{D}(\mu,x),\overline{D}(\mu,x)$ & &Lower and upper local dimension of $\mu$ at $x$ (Section \ref{subsection dimension and entropy1}). \\
\smallskip

$\mathcal{D},\mathcal{A},\mathcal{F}$ & &Partitions. \\

\smallskip

$\mathcal{D}_n(\R^d)$ (or $\mathcal{D}_n$) & &Partition of $\R^d$ into $n$-th level dyadic cubes (Section \ref{subsection dimension and entropy1}). \\

\smallskip

$\mathcal{D}_n(x),\mathcal{A}(x)$ & &The element of $\mathcal{D}_n$ (resp. $\mathcal{A}$) containing $x$. \\

\smallskip

$N_{2^{-n}}(A)$ & &The number of elements of $\mathcal{D}_n$ intersecting $A$ (Section \ref{subsection Partition and entropy 1}).\\

\smallskip

$\Lambda$ & &Alphabet set (finite). \\

\smallskip

$\sigma$ & &Shift map, $\sigma(x)_n=x_{n+1}$. \\

\smallskip

$[a]$ & &Cylinder set corresponding to $a\in \Lambda^n$. \\

\smallskip

$\mu^{[a]}$ & &$\sigma^n(\mu|_{[a]})/\mu([a])$ (Section \ref{subsection Dimension of CP-processes 11}). \\

\smallskip

$H(\mu,\mathcal{A})$ & &Shannon entropy (Section \ref{subsection dimension and entropy1}). \\

\smallskip

$\mathcal{A}_n^t$ & &Partition (Definition \eqref{eq: definition partition A-n-t 1}). \\
\smallskip

$\mu^{\mathcal{A}_n^t(z)}$ & &Definition \eqref{eq: def magnification measure 1}. \\

\hline
\end{tabular}

\section{Notation and preliminaries} \label{section notation}

\subsection{General notation and conventions} 

We use $\sharp A$ to denote the cardinality of a set $A$. In a metric space, $B(x,r)$ denotes the closed ball of radius $r$ around $x$. 

In this paper, a measure is always a Borel probability measure. The set of all Borel probability measures on a metric space $X$ will be denoted by $\mathcal{P}(X)$. Usually, we will not mention the $\sigma$-algebra of a measurable space; sets and functions are implicitly assumed to be Borel measurable when it is required. 

 If $X$ and $Y$ are metric spaces, and $f:X\to Y$ is any measurable map, then for any $\mu\in \mathcal{P}(X)$, we define $f\mu$ as the push-forward measure $\mu\circ f^{-1}$.
 
The topological support of a measure $\mu$ is denoted by ${\rm supp}(\mu)$; the restriction of $\mu$ on a set $E$ is denoted by $\mu|_E$.

We use $\delta_x$ to denote the Dirac measure at a point $x$.

We will use standard ``big O'' and ``little o'' notation.

\subsection{Symbolic space} 

In this subsection, we recall some classical notion for symbolic spaces.

Let $\Lambda$ be a finite set which we call an {\em alphabet set}. Let $\Lambda^\N$ be the symbolic space of infinite sequences from the alphabet set. We endow $\Lambda^\N$ with the standard metric $d_\rho$ with respect to a number $\rho\in (0,1)$:
\begin{equation}\label{eq: def distance simbolic 1}
d_\rho(x,y)=\rho^{\min\{n:x_n\neq y_n\}}.
\end{equation}
Then $(\Lambda^\N,d_\rho)$ is a compact totally disconnected metric  space. 

We denote by $\Lambda^*=\bigcup_{n\ge 0}\Lambda^n$  the set of finite words (with the convention that $\Lambda^0=\{\emptyset\}$). For $n\ge 0$, the {\em length} of a word $u\in \Lambda^n$, denoted by $|u|$, is defined to be $n$. For $u\in \Lambda^n$, the $n$-th level {\em cylinder} associated to $u$ is the set 
$$[u]=\{x\in \Lambda^\N: x_1\cdots x_n=u\}.$$ Every cylinder is a closed and open set. For $x\in \Lambda^* \cup \Lambda^\N$, we will use $$x_1^k=x_1\cdots x_k$$
to represent the word consisting of the $k$ first letters of $x$ when $k\le |x|$. Define the left-shift $\sigma$ on $\Lambda^\N$ by 
$$\sigma((x_n)_{n\ge 1})=(x_{n+1})_{n\ge 1}.$$

\subsection{Dimension and entropy} \label{subsection dimension and entropy1}

In this subsection, we recall some basic notion and facts about dimension and entropy of measures (or sets).

We use $\dim_{\rm H} A$ and $\overline{\dim}_{\rm B} A$ to denote the Hausdorff dimension and upper box-counting dimension of a set $A$, respectively.

\subsubsection{Dimension of measures}
Let $\mu$ be a Borel measure on a metric space. The {\em lower (Hausdorff) dimension} of $\mu$ is defined as 
$$\dim_*(\mu)=\inf\left\{\dim_{\rm H} A: \mu(A)>0\right\}.$$
Closely related to the lower dimension of $\mu$ is the {\em lower local dimension}, defined at each $x\in {\rm supp}(\mu)$ as 
$$\underline{D}(\mu,x)=\liminf_{r\to0}\frac{\log \mu(B(x,r))}{\log r}.$$
Similarly, we can consider the upper limit and define the upper local dimension $\overline{D}(\mu,x)$ of $\mu$ at $x$. When $\underline{D}(\mu,x)=\overline{D}(\mu,x)$, we say that the local dimension of $\mu$ at $x$ exists and denote it by $D(\mu,x)$. If the local dimension of $\mu$ exists and is constant $\mu$-almost everywhere, then $\mu$ is called {\em exact dimensional} and the almost sure local dimension is denoted by $\dim (\mu)$. For more details about different definitions of dimensions of measures, we refer the readers to \cite{Fan94, Falconerbook1,MattilaBook1,FanLauRao}.


\subsubsection{Partitions and entropy} \label{subsection Partition and entropy 1}
Let $\mu$ be a Borel measure on a metric space $X$. For a finite or countable partition $\mathcal{A}$ of $X$, the entropy of $\mu$ with respect to $\mathcal{A}$ is 
$$H(\mu,\mathcal{A})=-\sum_{A\in \mathcal{A}}\mu(A)\log \mu(A)$$
with the convention that $0\log0=0.$ {Here and in what follows, the logarithm is in base $e$.}

Next, we define entropy dimension of measures--first in the symbolic space, then in the Euclidean space.

In a symbolic space $(\Lambda^\N,d_\rho)$, let $\mathcal{F}_n$ be the partition of $\Lambda^\N$ given by the $n$-th level cylinder sets $\{[u]: u\in \Lambda^n\}$. For a set $A\subset \Lambda^\N$, we will use $N_{\rho^n}(A)$ to count the number of elements of $\mathcal{F}_n$ intersecting $A$.
For $\mu\in \mathcal{P}(\Lambda^\N)$, we define the {\em entropy dimension} of $\mu$ by 
$$\dim_e(\mu)=\lim_{n\to\infty}\frac{1}{-n\log \rho}H(\mu,\mathcal{F}_n),$$
if the limit exists; otherwise we consider the upper and lower entropy dimensions $\overline{\dim}_e(\mu)$ and $\underline{\dim}_e(\mu)$ defined by replacing limit, respectively, by $\limsup$ and $\liminf$.

Now, we define the entropy dimension on Euclidean space. For any $n\ge 0$, let $\mathcal{D}_n(\R^d)$ be the collection of $n$-th level {\em dyadic cubes}  of $\R^d$, that is, 
$$\mathcal{D}_n(\R^d):=\left\{\prod_{i=1}^d[\frac{k_i}{2^n},\frac{k_i+1}{2^n}): {(k_1,\ldots,k_d)}\in \Z^d\right\}.$$
Then $\mathcal{D}_n(\R^d)$ is a partition of $\R^d$.  For a set $A\subset \R^d$, we will use $N_{2^{-n}}(A)$ to count the number of elements of $\mathcal{D}_n(\R^d)$ intersecting $A$. 
For $\mu\in  \mathcal{P}(\R^d)$, the entropy dimension of $\mu$ is defined as 
$$\dim_e(\mu)=\lim_{n\to\infty}\frac{1}{n\log 2}H(\mu,\mathcal{D}_n(\R^d)),$$
if the limit exists; otherwise we consider the upper and lower entropy dimensions.
We will simply write $\mathcal{D}_n$ for $\mathcal{D}_n(\R^d)$ when no confusion can arise.



The following lemma presents some relationships between different dimensions of a measure.

\begin{lemma}\label{lemma different dimensions of a measure 1}
Let $\mu$ be a measure on $\R^d$ or $\Lambda^\N$. Then 
$$\dim_*(\mu)\le \underline{\dim}_e(\mu)\le \overline{\dim}_e(\mu).$$
If $\mu$ is exact dimensional, then 
$$\dim_*(\mu)=\dim (\mu)= \underline{\dim}_e(\mu)= \overline{\dim}_e(\mu).$$
\end{lemma}

\begin{proof}
Proofs for the Euclidean case can be found in \cite{FanLauRao}. The symbolic case is analogous.
\end{proof}

\subsubsection{Dimensions of product sets}
We recall the following dimension formula for dimensions of product sets.

\begin{lemma}[Theorem 8.10 of \cite{MattilaBook1}]\label{lemma Dimensions of product sets}
Let $E,F\subset\R^d$ be non-empty  Borel sets. Then 
$$\dim_{\rm H}E+\dim_{\rm H}F\le \dim_{\rm H}(E\times F)\le\overline{\dim}_{\rm B}(E\times F)\le  \overline{\dim}_{\rm B}E+\overline{\dim}_{\rm B}F.$$
\end{lemma}

\subsection{Dynamical systems} \label{preliminary on DS 1}

In this subsection, we collect some basic notions and properties of dynamical systems. We refer the reader to \cite{Waltersbook, Downarowicz2011book,EWbook} for more information.

\subsubsection{Measure preserving dynamical system} \label{subsection preliminary on DS 2}
 By a {\em Measure preserving dynamical system} (or {\em dynamical system} for short) we mean a quadruple $(X,\mathcal{B},T,\mu)$ where $X$ is a compact metric  space, $\mathcal{B}$ is the Borel $\sigma$-algebra on $X$, $T:X\to X$ is a Borel map and $\mu$ is a $T$-invariant measure. We shall often omit $\mathcal{B}$ in  our notation and abbreviate the system to $(X,T,\mu)$. 
 
 A dynamical system is {\em ergodic} if the only invariant sets are trivial, i.e., if $\mu(A\Delta T^{-1}A)=0$, then $\mu(A)=0 $ or $\mu(A)=1$. By the {\em ergodic decomposition theorem}, every $T$-invariant measure $\mu$ can be decomposed as mixtures of $T$-invariant ergodic measures:   {$\mu=\int \mu^{(x)} d\mu(x)$, where for $\mu$-a.e. $x$, $\mu^{(x)}$ is a $T$-invariant and ergodic measure, called an {\em ergodic component} of $\mu$.  We refer the reader to \cite[Chapter 4.2]{EWbook} for more information.}
 

Another important notion in ergodic theory is {\em weak-mixing}. For the precise definition of weak-mixing and its many equivalent formulations, see \cite{Waltersbook, EWbook}. We will make use the following characterization of weak-mixing (see\cite[Theorem 2.36]{EWbook}): a dynamical system $(X,T,\mu)$  is weakly mixing if and only if, for any ergodic dynamical system $ (Y, S, \nu)$, the product system $ (X\times Y,T\times S,\mu \otimes \nu)$ is also ergodic.

An important class of dynamical systems that we will have occasion to use are {\em symbolic dynamical systems}, in which $X$ is the symbolic space $\Lambda^\N$ and $T$ is the shift transformation $\sigma$, and  $\mu$ is a  shift-invariant measure. In the case when $\mu$ is a {product} measure determined by a probability vector $p=(p_i)_{i\in \Lambda}$ on $\Lambda$, we call $(\Lambda^\N,\sigma,\mu)$ a {\em {Bernoulli shift}}.

A dynamical system $(Y,S,\nu)$ is a {\em factor} of $(X,T,\mu)$ if there exists a measurable map $\pi: X\to Y$, called the factor map, which is {\em equivariant}, i.e., $\pi\circ T=S\circ \pi$ and $\pi\mu=\nu$.

Let $(X,T,\mu)$ be a dynamical system.  A point $x\in X$ is {\em generic} for $\mu$ if
$$\frac{1}{N}\sum_{n=0}^{N-1}\delta_{T^nx}\to  \mu \ \ {\rm as}\ N\to\infty$$
in the weak-* topology.   It follows from the ergodic theorem that if $\mu$ is ergodic then $\mu$-a.e. $x$ is generic for $\mu$.

\subsubsection{Measure-theoretic entropy}\label{subsubsection Measure  entropy1}
The {\em measure-theoretic entropy} of a dynamical system $(X,T,\mu)$ will be denoted by $h(\mu,T)$.  We refer the reader to \cite{Waltersbook, Downarowicz2011book} for precise definition of entropy and related material.

For a finite measurable partition $\mathcal{A}$ of $X$, we write $\mathcal{A}_n=\bigvee_{k=0}^{n-1}T^{-k}\mathcal{A}$ for the coarsest common refinement of $\mathcal{A},T^{-1}\mathcal{A},\cdots, T^{-(n-1)}\mathcal{A}$. We call $\{\mathcal{A}_n\}_{n\ge1}$ the filtration generated by $\mathcal{A}$ with respect to $T$. For each $n\ge 1$ and $x\in X$, $\mathcal{A}_n(x)$ is the unique element of $\mathcal{A}_n$ containing $x$. We use $\mathcal{A}_{\infty}=\bigvee_{k=0}^{\infty}T^{-k}\mathcal{A}$ to denote the $\sigma$-algebra generated by the partitions $\mathcal{A}_n$, $n\ge 1$. We say that $\mathcal{A}$ is a {\em generator}  for $T$ if $\mathcal{A}_{\infty}$ is the full Borel $\sigma$-algebra. 



\medskip

\section{CP-Processes}\label{section CP-process}

\subsection{General theory}

The CP-process theory was pioneered by Furstenberg in \cite{Furstenberg69},  initially as a tool to investigate Conjecture \ref{Intersection conjecture}.  Recently, a more systematic study of CP-processes was initiated by Furstenberg \cite{Furstenberg2008}, with further developments by Gavish \cite{Gavish}, Hochman \cite{Hochman2010}, Hochman and Shmerkin \cite{HS2012} and others. Let us first recall some basic concepts related to this theory in the symbolic setting.


Recall that  $\mathcal{P}(X)$ is the set of all Borel probability measures on a metric space $X$. A {\em distribution} is a Borel probability measure on $\mathcal{P}(X)$ (or even larger spaces). Notice that distributions are measures on space of measures. 

Fix a finite alphabet  $\Lambda$. For $0<\rho<1$, consider the symbolic space $\Lambda^\N$ endowed with the metric defined as \eqref{eq: def distance simbolic 1}.
 Let 
$$\Omega=\left\{(\mu,x)\in \mathcal{P}(\Lambda^\N)\times \Lambda^\N: x\in {\rm supp}(\mu)\right\}.$$

The CP-process theory studies the dynamical properties under the action of magnification of measures.

\begin{definition}[Magnification dynamics]\label{definition magnification operator}
We define the {\em magnification operator} $M:\Omega\to\Omega$ as 
$$M(\mu,x)=(\mu^{[x_1]},\sigma(x)),$$
where $\mu^{[x_1]}=\sigma(\mu|_{[x_1]})/\mu([x_1])$. 
\end{definition}

It is clear that $M(\Omega)\subset \Omega$ and $M$ is continuous.
For any distribution $P$ on $\Omega$ (i.e., $P\in \mathcal{P}(\Omega)$), we denote by $P_1$ its marginal on the measure coordinate. 
\begin{definition}[Adaptedness]\label{definition Adaptedness}
A distribution $P$ on $\Omega$ is called {\em adapted} if for every $f\in C(\mathcal{P}(\Lambda^\N)\times \Lambda^\N)$,
$$
\int f(\mu,x)dP(\mu,x)=\int \left(\int f(\mu,x)d\mu(x)\right) d P_1 (\mu).
$$
\end{definition} 

In other words, $P$ is adapted if, conditioned on the measure component being  $\mu$, the point component  $x$ is distributed according to $\mu$. 
In particular, if a property holds for $P$-a.e. $(\mu,x)$  and  $P$ is adapted, then this property holds for $P_1$-a.e. $\mu$ and $\mu$-a.e. $x$.

\begin{definition}[CP-distribution]
A distribution $P$ on $\Omega$ is a {\em CP-distribution} if it is $M$-invariant and adapted. In this case, we call the system $(\Omega,P,M)$ a CP-process.
\end{definition} 

A CP-distribution $P$ is {\em ergodic} if  the  measure preserving system $(\Omega,P,M)$ is ergodic in the usual sense. If it is not ergodic, then we can consider its ergodic decomposition.

\begin{proposition}\label{prop ergodic decomposition}
The ergodic components of a CP-distribution are adapted, in particular, they are ergodic CP-distributions.
\end{proposition}
A proof of this result is indicated in the remark following Proposition 5.1 of \cite{Furstenberg2008}. See also \cite[Proposition 22]{Shmerkin2011} and \cite[Theorem 1.3]{Hochman2010} for alternative proofs.

\subsection{Dimension and generic properties of CP-processes}\label{subsection Dimension of CP-processes 11}
In this subsection, we list some useful properties of CP-processes that we will use later.  The first one concerns dimension information of typical measures for  ergodic CP-distributions.

\begin{proposition}[Theorem 2.1 of \cite{Furstenberg2008}] \label{proposition exact dim and formula}

Let $P$ be an ergodic CP-distribution. Then $P_1$-almost every measure $\mu$ is exact dimensional with dimension
$$\dim\mu =\frac{1}{\log \rho^{-1}}\int -\log \nu[x_1]dP(\nu,x)=\frac{1}{\log \rho^{-1}}\int \sum_{i\in \Lambda}-\nu[i]\log \nu[i]dP_1(\nu).$$
\end{proposition}

For an ergodic CP-distribution $P$, we denote by $\dim P$ the almost sure dimension of $\mu$ for a $P$-typical $\mu$.

We will use several times the following lemma which is an immediate consequence of the ergodic theorem and the adaptedness property of CP-processes. We denote 
\begin{equation}\label{eq: definition rescaled measure symbolic 1}
\mu^{[x_1^n]}=\sigma^n(\mu|_{[x_1^n]})/\mu([x_1^n]).
\end{equation}

\begin{lemma} \label{lemma generating CP-distribution}
Let $P$ be an ergodic CP-distribution. Then $P_1$-a.e. $\mu$ generates $P_1$ in the  sense that  for $\mu$-a.e. $x$, we have
\begin{equation}\label{lemma generating CP-distribution 1}
\frac{1}{N}\sum_{n=0}^{N-1}\delta_{\mu^{[x_1^n]}} \to P_1 \ \ \textrm{ weak-* as } N\to \infty.
\end{equation}
\end{lemma}
For a measure $\mu$ which  generates $P_1$ in the above sense, we say $\mu$ is generic for   $P_1$.
As a corollary of Proposition \ref{proposition exact dim and formula} and Lemma \ref{lemma generating CP-distribution}, we obtain the following easy but useful  properties concerning typical measures of CP-distributions with positive dimension. Similar results have appeared in \cite{HS2015}.

\begin{proposition} \label{prop non-concentration and entropy}
Let $P$ be an ergodic CP-distribution with $\dim P=h>0$. For any $\epsilon >0$, there exists $n_0(\epsilon)\in \N$ such that for each $\mu$ which is generic for  $P_1$ and for $\mu$-a.e. $x$, 
\begin{equation}\label{eq: prop non-concentration 1}
\liminf_{N\to\infty}\frac{1}{N}\sharp \left\{1\le k\le N : \max_{u\in \Lambda^{n_0(\epsilon)}}\mu^{[x_1^k]}([u])\le \epsilon\right\}>1-\epsilon
\end{equation}
and 
\begin{equation}\label{eq: prop  large scale entropy 1}
\liminf_{N\to\infty}\frac{1}{N}\sharp \left\{1\le k\le N :  H(\mu^{[x_1^k]},\mathcal{F}_n)\ge n(h\log \rho^{-1}-\epsilon)\right\}>1-\epsilon \ \textrm{ for all  } n\ge n_0(\epsilon).
\end{equation}
In particular, for $P_1$-a.e. $\mu$ and $\mu$-a.e. $x$, the above properties hold. 
\end{proposition}

\begin{proof}
The proof is similar to that of  \cite[Lemma 4.11]{HS2015}. Fix any $\epsilon >0$. By Proposition \ref{proposition exact dim and formula}, $P_1$-a.e. $\nu$ is exact dimensional with dimension $h>0$, so $\nu$ is non-atomic and using Lemma \ref{lemma different dimensions of a measure 1} we have
$$\lim_{n\to\infty}\frac{1}{n}H(\nu,\mathcal{F}_n)=h\log\rho^{-1}.$$ 
Thus for $P_1$-a.e. $\nu$, there exists a finite integer $n(\nu)$ such that for each $n\ge n(\nu)$,
\begin{equation}\label{eq: prop non-concentration and entropy}
\max_{u\in \Lambda^n}\nu([u])< \epsilon \ \ \textrm{and}\ \ H(\nu,\mathcal{F}_n)> n(h\log \rho^{-1}-\epsilon).
\end{equation}
It follows that there exist a set $E_\epsilon$ of measures with $P_1(E_\epsilon)>1-\epsilon$ and a finite $n_0(\epsilon)\in \N$ such that $n_0(\epsilon)\ge n(\nu)$ for $\nu\in E_\epsilon$. For any $n\ge n_0(\epsilon)$, let $E_\epsilon^n$ be the set of measures $\nu$ such that \eqref{eq: prop non-concentration and entropy} holds. Then  $E_\epsilon\subset E_\epsilon^n$ and $E_\epsilon^n$ is open. Since  $\mu$ generates $P_1$, we have, for $\mu$-a.e. $x$,
$$\liminf_{N\to\infty}\frac{1}{N}\sum_{k=0}^{N-1}\delta_{\mu^{[x_1^k]}}(E_\epsilon^n)\ge P_1(E_\epsilon^n)>1-\epsilon.$$
The above statement holds for each $n\ge n_0(\epsilon)$, which is what we wanted to show.
\end{proof}

\begin{remark}\label{remark prop non-concentration and entropy}
In the above proof,  we saw  that the properties \eqref{eq: prop non-concentration 1} and \eqref{eq: prop  large scale entropy 1} hold for each pair $(\mu,x)$ satisfying \eqref{lemma generating CP-distribution 1}.
\end{remark}

\smallskip

\section{Constructions of CP-distributions based  on  $K=C_\alpha\times C_\beta$} \label{section Constructions of CP-distributions on K} 

Let us first recall the sets  $C_\alpha$ and $C_\beta$ and some preliminary results about them.
Fix two real numbers $0<\beta<\alpha<1$ such that $\theta=\log\alpha/\log \beta$ is irrational. Let $\Phi=\{\phi_i(x)=\alpha x +\lambda_i^\alpha\}_{i=1}^m$ and $\Psi=\{\psi_j(x)=\beta x +\lambda_i^\beta\}_{j=1}^l$ be two regular self-similar IFSs on $\R$. Let $C_\alpha$ be the attractor of $\Phi$ and $C_\beta$ be the attractor of $\Psi$. Let $K=C_\alpha\times C_\beta$.

In this section, assuming the existence of a  slice ${\ell}_0\cap K$ with upper box dimension $\gamma>0$, we construct a family of  ergodic CP-distributions  having dimensions at least $\gamma$ and supported on measures which are supported on slices of $K$. The construction of such CP-distributions is essentially due to Furstenberg \cite{Furstenberg69}, we just reinterpret the material in our setting. 

Since the IFSs $\Phi$ and $\Psi$ satisfy the convex open set condition, there exist open intervals $I_\alpha$ and $I_\beta$ with $\phi_i(I_\alpha)\subset I_\alpha$ ($1\le i\le m$) and  $\psi_j(I_\beta)\subset I_\beta$ ($1\le j\le l$) such that   $$\phi_{i_1}(I_\alpha)\cap \phi_{i_2}(I_\alpha)=\emptyset  \textrm{ for } i_1\neq i_2\  \textrm{ and } \   \psi_{j_1}(I_\beta)\cap \psi_{j_2}(I_\beta)=\emptyset  \textrm{ for } j_1\neq j_2.$$
Let $\{I_\alpha^i\}_{i=1}^m$ be a partition of $\bigcup_{i=1}^m\phi_i\left(\overline{I_\alpha}\right)$ such that each $I_\alpha^i$ is an interval which may be open, closed or half open and whose interior is the same as that of $\phi_i\left(\overline{I_\alpha}\right)$. Similarly, we choose such a partition $\{I_\beta^j\}_{j=1}^l$ for $\bigcup_{j=1}^l\psi_j\left(\overline{I_\beta}\right)$.
Then we define $S_\alpha$ to be the inverse map of $\Phi$ on $\bigcup_i\phi_i(I_\alpha)$, that is, the restriction of $S_\alpha$ on $I_\alpha^i$ is $\phi_i^{-1}$ for $1\le i\le m$. Let $S_\beta$ be the inverse map of $\Psi$ on $\bigcup_j\psi_j(I_\beta)$.
We define two maps on $(\bigcup_i\phi_i(I_\alpha))\times (\bigcup_j\psi_j(I_\beta))$ by 
$$\varphi_1(x,y)=(S_\alpha(x),y)\ \ \textrm{ and }\ \ \varphi_2(x,y)=(S_\alpha(x),S_\beta(y)).$$
Then $K=C_\alpha\times C_\beta$ is invariant under both maps $\varphi_1$ and $\varphi_2$. Given a line ${\ell}$ with slope $u$ which intersects $K$, then $\varphi_1$ transforms ${\ell\cap [0,1]^2}$ into  finitely many line segments, each with slope $\alpha u$ and $\varphi_2$ transforms ${\ell\cap [0,1]^2}$ into finitely many line segments, each with slope $\alpha u/\beta$. 

Now suppose that there exists a line $\ell$ that intersects $K$ in a set of upper box dimension $\gamma>0$. The same will be true for at least one of the lines of $\varphi_1(\ell)$ and for one of the lines of $\varphi_2(\ell)$.  We can continue in this way and finally we will find a family $L$ of infinitely many lines such that each line of $L$ intersects $K$ in a set of upper box dimension $\gamma$. If the initial line $\ell$ has slope $u$ with $u\notin\{ 0,\infty\}$, then for each pair $(n,m)\in \N^2$ with $n\ge m$, there exists a line in $L$ with slope $u\alpha^n/\beta^m$. Since $\log \alpha/\log \beta$ is irrational, the set $\{u\alpha^n/\beta^m: n\ge m\}$ is dense in $(0,+\infty)$ or in $(-\infty,0)$ depending on whether $u>0$ or $u<0$. 

 In the rest of this paper, we always make the assumption that 
\begin{equation}\label{eq: assumption slice dim}
\textit{there exists a line $\ell_0$  with slope $u_0\in (0,+\infty)$  such  that  $\overline{\dim}_{\rm B}(\ell_0\cap K)=\gamma>0$.}
\end{equation}
Our ultimate aim is to show that, in this case, we must have $\dim_{\rm H} K\ge 1+\gamma.$
{For the case of negative slope $u_0$, we apply a reflection to $C_\alpha$ to make the slope positive.}

In the rest of this section, we will follow Furstenberg \cite{Furstenberg69} to construct an ergodic CP-distribution (with dimension $\gamma$) on the space of measures which are supported on  slices of $K$ with slopes in $[1,1/\beta]$.  In the end of Subsection \ref{subsection construction of CP-dist 1}, as a direct application of this CP-distribution, we will give the proof of Furstenberg's main result in \cite[Theorem 4]{Furstenberg69}: under the assumption \eqref{eq: assumption slice dim}, for Lebesgue almost all $u\in (0,+\infty)$, there exists a slice of $K$ with slope $u$ and  Hausdorff dimension $\ge \gamma$.


\subsection{Symbolic setting} 

Let $\Lambda_\alpha=\{\lambda_i^\alpha\}_{i=1}^m$ and $\Lambda_\beta=\{\lambda_j^\beta\}_{j=1}^l$. Note that $C_\alpha$ can be written as 
$$C_\alpha=\left\{\sum_{n=1}^\infty\alpha^{n-1}a_n: (a_n)_{n\ge1}\in \Lambda_\alpha^\N\right\}.$$
A similar representation holds for $C_\beta$, replacing $\alpha$ by  $\beta$ and $ \Lambda_\alpha$ by $ \Lambda_\beta$.

Write $\Lambda=\Lambda_\alpha\times \Lambda_\beta$. Let $X=\Lambda^\N$. Recall that $\theta=\log \alpha/\log \beta$.  For each $t\in [0,1)=\R/\Z$, we construct a tree $X_t\subset X=\Lambda^\N$ as follows. For $s\in [0,1)$,  write $L(s)=\Lambda$ if $s\in [0,\theta)$ and $L(s)=\Lambda_\alpha\times \{\lambda_1^\beta\}$ otherwise. We define
$$R_\theta(s)=s-\theta \mod 1 \ \ \textrm{ for } s\in [0,1).$$
In the rest of this paper, we identify $[0,1)$ with $\R/\Z$, thus $[0,1)$ is compact and $R_\theta$ is continuous on it.

Let 
$$X_t=\prod_{n=0}^\infty L(R^n_\theta(t)).$$
By definition, for $x\in X_t$, the shifted point $\sigma(x)$ is an element of $X_{R_\theta(t)}$. On each $X_t$ we consider the metric $d_\alpha$ (recall \eqref{eq: def distance simbolic 1}).

For $s\in [0,1)$, let $Z(s)=\{n\ge0: R_\theta^n(s)\in [0,\theta)\}$. We write the elements of $Z(s)$ in an increasing order as $w_1(s)<w_2(s)<\cdots$. We define a projection map $\pi_t:X_t\to K$ by
$$\pi_t\left((a_n)_n,(b_n)_n\right)=\left(\sum_{n=1}^\infty\alpha^{n-1} a_n, \sum_{n=1}^\infty\beta^{n-1} b_{w_n(t)}\right).$$
Note that $\pi_t$ is a surjective map.

Let us record for later use some properties about $X_t$ and $\pi_t$ in the following lemma. We use ${\rm cov}_{r}(A)$ to denote the minimal number of balls of diameter $r$ needed to cover a set $A$. Recall also the notation $N_{\alpha^k}(A)$ (see Section \ref{subsection Partition and entropy 1}).
\begin{lemma}\label{lemma basic propeties 1}
 \begin{itemize}
\item[(1)]  If $t_k, t\in [0,1)$ are such that  $t_k\to t$ and $R^n_\theta(t)\neq \theta$ for all $n$,  then $X_{t_k}\to X_t$ (under the Hausdorff metric) and $\pi_{t_k}\to \pi_{t}$.

\item[(2)]  There exists a constant $C_1>0$ such that the maps $\pi_t$ are uniformly $C_1$-Lipschitz.

\item[(3)]  There exists a constant $C_2>0$ such that for all $t\in [0,1)$ and all $A\subset X_t$, $N_{\alpha^k}(A)\le C_2\cdot{\rm cov}_{\alpha^k}(\pi_t(A))$ for each  $k\in \N$.
\item[(4)] For all $t\in [0,1)$ and all $A\subset X_t$, we have $\dim_{\rm H}A=\dim_{\rm H}\pi_t(A) $.
\end{itemize}
\end{lemma}
\begin{proof}

We give the proof for parts (1) and (2), the other parts are obvious. The first part follows from the fact that, if $R^n_\theta(t)\neq \theta$ for all $n\le M$, then for all sufficiently large $k$ we have $w_n(t_k)=w_n(t)$ for all $n\le M$, and this implies that the first $M$ generations of the trees $X_{t_k}$ and $X_t$ coincide, and that $\pi_{t_k}$ is uniformly close to $\pi_t$.

To prove part (2), it suffices to show that there is $C_1$ (independent of $t$) such that 
$$C_1^{-1}\alpha^k\le \beta^{r_k(t)}\le C_1\alpha^{k},$$
where $r_k(t)=\sharp\{0\le i\le k-1: R^i_\theta(t)\in [0,\theta)\}$. This is equivalent to say that $|r_k(t)-k\theta|$ is bounded by some uniform constant. To show this, we only need to observe that $r_k(t)$ is the number of  $i\in \{0,\ldots,k-1\}$ such that  there exists an integer $n$ with $t-i\theta\ge n$ and $t-(i+1)\theta<n$. Thus $-r_k(t)$ is the largest integer not greater than $t-k\theta$, from which we deduce that $|r_k(t)-k\theta|\le 2$.

\end{proof}

Let $\ell_{u,z}$ denote the line through $z$ with slope $u$. We define
$$\mathcal{F}=\left\{(A,x,t):  t\in [0,1), A \subset X_t \textrm{ is compact, } x\in A, \pi_t(A)\subset K\cap \ell_{\beta^{-t},\pi_t(x)}\right\}.$$
Note that for any line $\ell_{\beta^{-t},z}$ with $t\in [0,1), z\in K$ and any $x\in \pi_t^{-1}(z)$, the set $( \pi_t^{-1}(K\cap \ell_{\beta^{-t},z}),x,t)\in \mathcal{F}$.

\begin{lemma}\label{lemma basic propeties 2}
\begin{itemize}
\item[(1)]  If $(A,x,t)\in \mathcal{F}$, then $(\sigma(A\cap [x_1]),\sigma(x),R_\theta(t))\in \mathcal{F}$.

\item[(2)]  Suppose $(A_k,y_k,t_k)\to (A,x,t)$ and $(A_k,y_k,t_k)\in \mathcal{F}$ for each $k$. If $R^n_\theta(t)\neq \theta$ for all $n$,  then $(A,x,t)\in \mathcal{F}$.
\end{itemize}
\end{lemma}

\begin{proof}
Note that for $x'\in X_t$, we have $\pi_{R_\theta(t)}(\sigma(x'))=\Phi_t(\pi_t(x'))$. Thus we have 
$$\pi_{R_\theta(t)}(\sigma(A\cap [x_1]))=\Phi_t(\pi_t(A\cap [x_1])).$$
From this we deduce the claim (1). The claim (2) is a consequence of part (1) of Lemma \ref{lemma basic propeties 1}. 
\end{proof}

\subsection{Construction of CP-distributions} \label{subsection construction of CP-dist 1}

Consider the space
$$Y=\mathcal{P}(X)\times X\times [0,1). $$
Note that $Y$ is a compact space.
We define a map $\hat{M}$ on $Y$ by
$$\hat{M}(\mu,x,t)=(\mu^{[x_1]},\sigma(x),R_\theta(t)).$$
The map $\hat{M}$ can be viewed as an ``extension" of the magnification operator $M$ in Definition \ref{definition magnification operator}.
It is continuous on $Y$ (where we consider the weak topology on $\mathcal{P}(X)$).

By the assumption \eqref{eq: assumption slice dim} and the discussion preceding it, there exist some $t_0\in [0,1)$ and a line $\ell$ with slope $\beta^{-t_0}$ such that $\overline{\dim}_{\rm B}  K\cap \ell=\gamma>0$. Let $E=\pi_{t_0}^{-1}(K\cap \ell)$. Then by parts (2) and (3)  of Lemma \ref{lemma basic propeties 1}, we have $\overline{\dim}_{\rm B} E=\gamma$ (in the space $X_{t_0}$).
Thus there exists a sequence $n_k\nearrow \infty$ such that
\begin{equation}\label{eq: upper box dim A1}
 \lim_{k\to\infty}\frac{\log N_{\alpha^{n_k}}(E)}{-n_k\log \alpha}=\gamma.
\end{equation}
We define a sequence of measures $\{\mu_k\}_k$ on $E$ by setting
$$\mu_k=\frac{1}{N_{\alpha^{n_k}}(E)}\sum_{u\in \Lambda^{n_k}: [u]\cap E\neq \emptyset} \delta_{x_u},$$
where $x_u$ is some point in $[u]\cap E$. Finally,  let
$$
P_k  =  \frac{1}{N_{\alpha^{n_k}}(E)}\sum_{u\in \Lambda^{n_k}: [u]\cap E\neq \emptyset} \delta_{(\mu_k,x_u,t_0)}
$$
and
$$
Q_k  =  \frac{1}{n_k}\sum_{i=0}^{n_k-1} \hat{M}^iP_k.
$$
By the construction of $P_k$, it is clear that for any $f\in C(Y)$, we have 
$$
\int f(\mu,x,t)dP_k(\mu,x,t)=\int \left(\int f(\mu,x,t)d\mu(x)\right) d (P_k)_{1,3} (\mu,t),
$$ 
where we use $(P_k)_{1,3}$ to denote the marginal of $P_k$ on the first and third coordinates. The same is true for $Q_k$. Let us call a distribution $P\in \mathcal{P}(Y)$ {\em globally adapted} if it satisfies the above identity. It follows from the definition that if a property holds for $P$-a.e. $(\mu,x,t)$  and  $P$ is globally adapted, then this property holds for $P_{1,3}$-a.e. $(\mu,t)$ and $\mu$-a.e. $x$.
Clearly, for a globally adapted distribution, its marginal on the first two coordinates $(\mu,x)$ is adapted in the sense of Definition \ref{definition Adaptedness}. For each $P\in \mathcal{P}(Y)$, we define
$$H(P)=\int \frac{1}{\log \alpha}\log \mu[x_1]dP_{1,2}(\mu,x),$$
where $P_{1,2}$ is the marginal of $P$ on  $(\mu,x)$. Let us calculate 
\begin{eqnarray*}
H(Q_k) & = & \frac{1}{n_k}\frac{1}{N_{\alpha^{n_k}}(E)} \sum_{u\in \Lambda^{n_k}: [u]\cap E\neq \emptyset}\sum_{i=1}^{n_k}\frac{1}{\log \alpha}\log\frac{\mu_k[u_1^i]}{\mu_k[u_1^{i-1}]} \\
 & = & \frac{1}{n_k}\frac{1}{N_{\alpha^{n_k}}(E)} \sum_{u\in \Lambda^{n_k}: [u]\cap E\neq \emptyset}\frac{1}{\log \alpha}\log\mu_k[u]=\frac{\log N_{\alpha^{n_k}}(E)}{-n_k\log \alpha}.
\end{eqnarray*}
It follows from \eqref{eq: upper box dim A1} that 
$$H(Q_k)\to \gamma \ \textrm{  as } k\to\infty.$$
Passing to a  further subsequence we can assume that $Q_k\to Q$ in $\mathcal{P}(Y)$. Now by continuity of $\hat{M}$, $Q$ is $\hat{M}$-invariant; and since each $Q_k$ is globally adapted, we deduce that $Q$ is also  globally  adapted. Thus the marginal of $Q$ on  $(\mu,x)$ is a CP-distribution.
Since the map $H$ is continuous on $\mathcal{P}(Y)$, we have
$$H(Q)=\lim_{k\to\infty}H(Q_k)=\gamma.$$

Let $$Q=\int Q^{(\mu,x,t)}d Q(\mu,x,t)$$ be the ergodic decomposition of $Q$. We define
\begin{equation}\label{eq: def mathcal E gamma}
\mathcal{E}_\gamma=\left\{(\mu,x,t)\in Y: H(Q^{(\mu,x,t)})\ge \gamma\right\}.
\end{equation}
Then we have $Q(\mathcal{E_\gamma})>0$ and for $Q$-a.e. $(\mu,x,t)\in \mathcal{E}_\gamma$, the marginal of $Q^{(\mu,x,t)}$ on the first two coordinates, denoted by $Q^{(\mu,x,t)}_{1,2}$, is an ergodic CP-distribution with dimension $H(Q^{(\mu,x,t)})\ge \gamma$. Note that for the adaptedness of $Q^{(\mu,x,t)}_{1,2}$, we have used Proposition \ref{prop ergodic decomposition}.

Let 
$$\Xi_{\mathcal{F}}=\bigcup_{(A,x,t)\in \mathcal{F}}\mathcal{P}(A)\times \{x\}\times\{t\}.$$

\begin{lemma}\label{lemma full distribution of Xi-F}
The distribution $Q$ is supported on $\Xi_{\mathcal{F}}$. In particular, this holds for $Q$-a.e. ergodic component of $Q$.
\end{lemma}
\begin{proof}
We need to prove that for $Q$-a.e. $(\mu,x,t)$ we have $({\rm supp}(\mu),x,t)\in \mathcal{F}$. Since $Q$ is a weak limit of $Q_k$ and each $Q_k$ is supported on $\Xi_{\mathcal{F}}$, it follows that $Q$ is supported on triples of the form 
$$(\mu,x,t)=\lim_{k\to\infty}(\mu_k,x_k,t_k)$$
with $({\rm supp}(\mu_k),x_k,t_k)\in \mathcal{F}$. Now, since the marginal of $Q$ on the third coordinate is an $R_\theta$-invariant measure on $[0,1)$, it must be Lebesgue measure. Thus for $Q$-a.e. $(\mu,x,t)$, we have $R_\theta^n(t)\neq \theta$ for all $n$. From this, part (2) of Lemma \ref{lemma basic propeties 2} and the fact that ${\rm supp}(\mu)\subset \liminf_{k\to\infty}{\rm supp}(\mu_k)$, we deduce that $({\rm supp}(\mu),x,t)\in \mathcal{F}$.
\end{proof}

We finish this subsection by giving the proof of the following result of Furstenberg \cite[Theorem 4]{Furstenberg69} 
by using the CP-distributions $\{Q^{(\mu,x,t)}_{1,2}\}_{(\mu,x,t)}$ we constructed above.

\begin{theorem}[Furstenberg, \cite{Furstenberg69}]\label{thm Furstenberg result 69}
Assume that \eqref{eq: assumption slice dim} hold. Then for Lebesgue almost every $u\in (0,+\infty)$, there exists a line $\ell$ with slope $u$ such that $\dim_{\rm H}\ell\cap K\ge \gamma.$
\end{theorem}

\begin{proof}
By the discussion preceding assumption \eqref{eq: assumption slice dim}, we only need to show that for Lebesgue almost every $u\in [1,\beta^{-1}]$, there exists a line $\ell$ with slope $u$ such that $\dim_{\rm H}\ell\cap K\ge \gamma.$  Let $Q, \mathcal{E}_\gamma,\Xi_{\mathcal{F}}$ be as above. We choose an element $(\mu,x,t)\in \mathcal{E}_\gamma$ such that the ergodic component $Q^{(\mu,x,t)}$ is supported on $\Xi_{\mathcal{F}}$ and its marginal $Q^{(\mu,x,t)}_{1,2}$ is an ergodic CP-distribution (with dimension at least $ \gamma$). Thus for  $Q^{(\mu,x,t)}$-a.e. $(\vartheta, y,s)$, $\vartheta$ is a measure with dimension at least $ \gamma$. Again, since the marginal of $Q^{(\mu,x,t)}$ on the third coordinate is an $R_\theta$-invariant measure on $[0,1)$, it must be Lebesgue measure.
Hence for Lebesgue almost every $s\in [0,1)$, there exists $(\vartheta,y)$ such that $(\vartheta,y,s)\in \Xi_{\mathcal{F}}$ and $\dim \vartheta\ge \gamma$. From the definition of $\Xi_{\mathcal{F}}$ and  part (4) of Lemma \ref{lemma basic propeties 1}, we deduce that there exists a line $\ell$ with slope $\beta^{-s}$ such that $\dim_{\rm H} \ell\cap K\ge \gamma.$

\end{proof}


\smallskip

\section{A skew product $U$ on $K\times [0,1)$ and a class  of  $U$-invariant measures }\label{subsection special measure 1}

In the previous section, we have constructed a family of ergodic $\hat{M}$-invariant distributions $\{Q^{(\mu,x,t)}\}_{(\mu,x,t)\in \mathcal{E}_\gamma}$ whose marginals on the first two coordinates are ergodic CP-distributions  having dimensions at least $ \gamma$ and supported on  measures which are supported on slices of $K$. 
In Subsection \ref{subsection The transformation $U$ and its basic properties}, we will define a skew product on $K\times [0,1)$, which can be regarded as the geometric version of the shift map $\sigma$ on $X_t$ ($t\in [0,1)$), and we study some partitions generated by $U$. In Subsection \ref{subsection construction of a class of U-inv measures}, we will  construct a family of $U$-invariant measures such that each of them is a certain form of  superposition of measures distributed according to  $Q^{(\mu,x,t)}$ with some $(\mu,x,t)\in \mathcal{E}_\gamma$. In Subsection \ref{subsection Properties and entropy of the measure},   we will study  some further properties of  such a $U$-invariant measure.

\subsection{The transformation $U$ and some basic properties}\label{subsection The transformation $U$ and its basic properties}

 For each $t\in [0,1)$, we define a map $\Phi_t:K\to K$ by 
\begin{equation}\label{eq: definition Phi_t}
\Phi_t(x,y) = \left\{ \begin{array}{ll}
(S_\alpha(x),S_\beta(y)) & \textrm{if $t\in [0,\theta)$}\\
(S_\alpha(x),y) & \textrm{otherwise.}
\end{array} \right.
\end{equation}
Note that, by the discussion about $S_\alpha$ and $S_\beta$ at the beginning of Section \ref{section Constructions of CP-distributions on K}, we have the following result.

\begin{lemma}\label{lemma change of slope when apply Phi-t 1}
If $\ell$ is a line with slope $\beta^{-t}$ ($t\in [0,1)$) which intersects $K$, then $\Phi_t(\ell)$ consists of a finite number of lines, each of which has slope $\beta^{-R_\theta(t)}$. 
\end{lemma}

We consider the following transformation $U:K\times [0,1)\to K\times [0,1)$ defined as a skew product
$$U(z,t)=(\Phi_t(z),R_\theta(t)).$$
Recall that $\Phi_t$ is defined by \eqref{eq: definition Phi_t} and $R_\theta$ is the irrational rotation map defined by $R_\theta(t)=t-\theta \mod 1.$ 

Let us write $U^n_t(z)$ for the first component of $U^n(z,t)$. Then it follows from the definition of $U$ that we have $$U_t^n(z)=\Phi_{R^{n-1}_\theta(t)}\circ\cdots\circ\Phi_{R_\theta(t)}\circ\Phi_{t}(z)=(S_\alpha^n(z_1),S_\beta^{r_n(t)}(z_2)) \ \textrm{ for } z=(z_1,z_2),$$
where $r_n(t):=\sharp\{0\le k\le n-1: R^{k}_\theta(t)\in [0,\theta) \}$.

In the following, we define a sequence of refining partitions of $K\times [0,1)$, which is generated by $U$. First, recall that  $\{I_\alpha^i\}_{i=1}^m$ and $\{I_\beta^j\}_{j=1}^l$ are, respectively,  partitions of $\bigcup_{i=1}^m\phi_i(\overline{I_\alpha})$ and  $\bigcup_{j=1}^l\psi_j(\overline{I_\beta})$ (see the beginning of Section \ref{section Constructions of CP-distributions on K}). We take $\mathcal{C}=\{[0,\theta),[0,1)\setminus [0,\theta) \}$ as a partition of $[0,1)$. Let 
\begin{equation}\label{eq: def of partition level 1}
\mathcal{B}_1=\{I_\alpha^i\cap C_\alpha\}_{i=1}^m\times \{I_\beta^j\cap C_\beta\}_{j=1}^l\times \mathcal{C}
\end{equation}
be our first level partition of $K\times [0,1)$. Then for $n\ge 2$, let 
$$\mathcal{B}_n=\bigvee_{k=0}^{n-1}U^{-k}(\mathcal{B}_1).$$ 

For later use, let us give some more details about the partitions $\{\mathcal{B}_n\}_n$. For $n\ge 1$, let 
$$\mathcal{C}_n=\bigvee_{k=0}^{n-1}R_\theta^{-k}(\mathcal{C}).$$
Recall that the map $U_t^k$ is defined via the relation $U^k(z,t)=(U_t^k(z),R_\theta^k(t))$.   For $n\ge 1$ and $t\in [0,1)$, let 
$$\mathcal{A}_n^t=\bigvee_{k=0}^{n-1}(U_t^{k})^{-1}\left(\{I_\alpha^i\cap C_\alpha\}_{i=1}^m\times \{I_\beta^j\cap C_\beta\}_{j=1}^l\right).$$
Note that by the fact $U_t^n(z)=(S_\alpha^n(z_1),S_\beta^{r_n(t)}(z_2))$, we have 
\begin{equation}\label{eq: definition partition A-n-t 1}
\mathcal{A}_n^t=\left(\bigvee_{k=0}^{n-1}S_\alpha^{-k}\left(\{I_\alpha^i\cap C_\alpha\}_{i=1}^m\right)\right)\times \left(\bigvee_{k=0}^{n-1}S_\beta^{-r_k(t)}\left(\{I_\beta^j\cap C_\beta\}_{j=1}^l\right)\right) .
\end{equation}
Thus by the definition of $\{r_k(t)\}_k$, we have $\mathcal{A}_n^t=\mathcal{A}_n^{t'}$ if $t$ and $t'$ both belong to a same element of $\mathcal{C}_n$.
By the definition of $U$, it is not hard to check that each element of $\mathcal{B}_n$ has the form $A\times C$ with some $C\in \mathcal{C}_n$ and $A\in\mathcal{A}_n^t$ for some $t\in C$. 


As usual, for all $z\in K$, we write $\mathcal{A}_n^t(z)$ for the unique element of $\mathcal{A}_n^t$ containing $z$. 
For any measure $\nu\in \mathcal{P}(K)$ and $z\in {\rm supp}(\nu)$, we write 
\begin{equation}\label{eq: def magnification measure 1}
\nu^{\mathcal{A}_n^t(z)} =  U^n_t\left(\frac{\nu|_{\mathcal{A}_n^t(z)}}{\nu (\mathcal{A}_n^t(z))}\right).
\end{equation}
 Note that if $\nu\in \mathcal{P}(\ell\cap K)$ for some line $\ell$ with slope $\beta^{-t}$, then  $\nu^{\mathcal{A}_n^t(z)}\in \mathcal{P}(\ell'\cap K)$ for some line $\ell'$ with slope $\beta^{-R_\theta^n(t)}$. Recall also the notation $\mu^{[x_1^n]}$ (see \eqref{eq: definition rescaled measure symbolic 1}).
In what follows, the boundary of $\mathcal{A}_n^t$ should be understood as relative to the space $K$.
\begin{lemma}\label{lemma basic partition 2}
\begin{itemize}
\item[(1)] Let $t\in [0,1)$ and $x\in X_t$. If $\pi_t(x)$ is not at the boundary of $\mathcal{A}_n^t(\pi_t(x))$, then the set $\pi_t([x_1^n])$ coincides with $\mathcal{A}_n^t(\pi_t(x))\cap K$ except possibly at the boundary points of $\mathcal{A}_n^t(\pi_t(x))$. 

\item[(2)] Let $(\mu,x,t)\in \Xi_{\mathcal{F}}$. If $\mu $ is non-atomic, then for $\mu$-a.e. $x$ and $n\ge1$, we have
\begin{equation}\label{eq: special measure 8}
\pi_{R_\theta^n(t)}\left(\mu^{[x_1^n]}\right)=(\pi_t\mu)^{\mathcal{A}_n^t(\pi_t(x))}.
\end{equation}
\end{itemize}
\end{lemma}
\begin{proof}
The part (1) is clear, we only need to prove (2). By definition,  $\pi_t\mu$ is a measure supported on some slice of $K$ with the form $K\cap \ell_{\beta^{-t},z}$ for some $z\in K$. It is clear that, for all $n\ge1$ and each element $A$ of $\mathcal{A}_n^t$, the support of  $\pi_t\mu$ intersects the boundary of $A$ in at most two points. Since $\mu $ is non-atomic, it follows that $\pi_t\mu$ gives zero measure to the boundary of $A$. Thus for $\mu$-a.e. $x$ and $n\ge1$, 
$$\pi_t(\mu|_{[x_1^n]})=\pi_t\mu|_{\mathcal{A}_n^t(\pi_t(x))}.$$
Note that for $t\in [0,1)$ and $x\in X_t$, we have 
$$U^n(\pi_t(x),t)=(U_t^n(\pi_t(x)),R_\theta^n(t))=(\pi_{R_\theta^n(t)}(\sigma^n(x)),R_\theta^n(t)).$$
Combining the above conclusions, we obtain \eqref{eq: special measure 8}.
\end{proof}

\subsection{Construction of a class of $U$-invariant measures}
\label{subsection construction of a class of U-inv measures}

This subsection is devoted to the construction of a class of $U$-invariant measures. We will first define these measures and then show that they are $U$-invariant.

Let $Q$ be the $\hat{M}$-invariant distribution constructed in  Subsection \ref{subsection construction of CP-dist 1}. 
Recall that $Q=\int Q^{(\mu,x,t)}d Q(\mu,x,t)$ is the ergodic decomposition of $Q$.
By the ergodic theorem, for $Q$-a.e. $(\mu,x,t)$,  the triple $(\mu,x,t)$ generates $Q^{(\mu,x,t)}$ in the sense that
\begin{equation}\label{eq: special measure 1}
\frac{1}{N}\sum_{n=0}^{N-1}\delta_{\hat{M}^n(\mu,x,t)}\to Q^{(\mu,x,t)} \ \textrm{   as } N\to\infty
\end{equation}
in the weak-* topology.
Consider the map $G:\Xi_{\mathcal{F}}\to \mathcal{P}(K\times [0,1))$ defined by
$$G(\mu,x,t)=\pi_t\mu\times \delta_t.$$
Then $G$ is continuous. It follows from \eqref{eq: special measure 1} that 
for $Q$-a.e. $(\mu,x,t)$, 
\begin{equation}\label{eq: special measure 2}
\frac{1}{N}\sum_{n=0}^{N-1}G(\hat{M}^n(\mu,x,t))\to \int GdQ^{(\mu,x,t)} \ \textrm{ as } N\to\infty. 
\end{equation}
Recall that by the definition of $\hat{M}$, we have 
\begin{equation}\label{eq: special measure 3}
\hat{M}^n(\mu,x,t)=\left(\mu^{[x_1^n]},\sigma^n(x),R_\theta^n(t)\right).
\end{equation}

We use $Q_{1,3}$ and $Q^{(\mu,x,t)}_{1,3}$ to denote, respectively, the marginals of $Q$ and $Q^{(\mu,x,t)}$ on the first and third coordinates. Recall the definition of $\mathcal{E}_\gamma$ (see \eqref{eq: def mathcal E gamma}).
\begin{lemma}\label{lemma generic measure convergence 1}
For $Q_{1,3}$-a.e. $(\mu,t)$ and $\mu$-a.e. $x$ with $(\mu,x,t)\in \mathcal{E}_\gamma$, we have
\begin{equation}\label{eq: special measure 9}
\frac{1}{N}\sum_{n=0}^{N-1}(\pi_t\mu)^{\mathcal{A}_n^t(\pi_t(x))}\times \delta_{R_\theta^n(t)}\to \nu^{(\mu,x,t)}:=\int \pi_s\vartheta\times \delta_{s} \ dQ^{(\mu,x,t)}_{1,3}(\vartheta,s)  \  \textrm{ as } N\to\infty
\end{equation}
in the  weak-* topology.
\end{lemma}

\begin{proof}
First, we claim that for $Q$-a.e. $(\mu,x,t)\in \mathcal{E}_\gamma$, the measure $\mu$ is non-atomic. To see this, recall that for $Q$-a.e. $(\mu,x,t)\in \mathcal{E}_\gamma$, the triple $(\mu,x,t)$ generates $Q^{(\mu,x,t)}$, and the marginal $Q^{(\mu,x,t)}_{1,2}$ is an ergodic CP-distribution with positive dimension. Let us fix any such  $(\mu,x,t)\in \mathcal{E}_\gamma$. Then $(\mu,x)$ generates the marginal $Q^{(\mu,x,t)}_{1,2}$, and it follows from Proposition \ref{prop non-concentration and entropy}, \eqref{eq: prop non-concentration 1} (and Remark \ref{remark prop non-concentration and entropy}) that $\mu$ is non-atomic.

Now, combining \eqref{eq: special measure 2}, \eqref{eq: special measure 3},  \eqref{eq: def magnification measure 1} and  part (2) of Lemma \ref{lemma basic partition 2}, we get \eqref{eq: special measure 9} for $Q$-a.e. $(\mu,x,t)\in \mathcal{E}_\gamma$. Since $Q$ is globally adapted, we deduce that \eqref{eq: special measure 9} holds for $Q_{1,3}$-a.e. $(\mu,t)$ and $\mu$-a.e. $x$ such that $(\mu,x,t)\in \mathcal{E}_\gamma$.
\end{proof}

We saw in the above proof that  formula \eqref{eq: special measure 9} actually holds  for $Q_{1,3}$-a.e. $(\mu,t)$ and $\mu$-a.e. $x$ with $H(Q^{(\mu,x,t)})>0$, but we
will not use this fact.



The rest of this subsection is devoted to the proof of the following.

\begin{proposition}\label{proposition invariance 1}
For $Q_{1,3}$-a.e. $(\mu,t)$ and $\mu$-a.e. $x$ with $(\mu,x,t)\in \mathcal{E}_\gamma$, the measure $\nu^{(\mu,x,t)}$ is $U$-invariant.
\end{proposition}

Our idea for the  proof of Proposition \ref{proposition invariance 1} is inspired by \cite[Theorem 2.1]{HS2015}
where it is shown, for a Borel map $T$ of a compact metric space $X$, how to relate the small-scale structure of a measure $\upsilon\in \mathcal{P}(X)$ to the distribution of $T$-orbits of $\upsilon$-typical points.

The proof of Proposition \ref{proposition invariance 1} relies on  three lemmas.
For any $(z,t)\in K\times [0,1)$, we define a sequence of measures 
$$\eta_N(z,t)=\frac{1}{N}\sum_{n=0}^{N-1}\delta_{U^n(z,t)}, \ \ N\ge1.$$ 
The first lemma shows that, for a given measure $\upsilon\in \mathcal{P}(K)$, when restricted on the elements of $\mathcal{B}_k , k\ge1$, the measures $\eta_N(z,t)$ and the Ces\`aro averages of $\upsilon^{\mathcal{A}_n^t(z)}\times \delta_{R^n_\theta(t)}$ are asymptotically the same for $\upsilon$-a.e. $z$.

\begin{lemma}\label{lemma invariance 1}
Let $\upsilon\in \mathcal{P}(K)$.
For any $t\in [0,1),$ $k\ge 1$ and  each $B\in \mathcal{B}_k$, we have 
$$\lim_{N\to\infty}\frac{1}{N}\sum_{n=0}^{N-1}\left(1_B(U^n(z,t))-\upsilon^{\mathcal{A}_n^t(z)}\times \delta_{R^n_\theta(t)}(B)\right)=0 \ \textrm{ for $\upsilon$-a.e. $z$.} $$ 
\end{lemma}

\begin{proof}
Fix $k\ge1$ and let $B\in \mathcal{B}_k$.
Recall that we can write $B=A\times C$ with some $C\in \mathcal{C}_k$ and $A\in \mathcal{A}_k^{t'}$ for some $t'\in C$. Then $1_B(U^n(z,t))=1_A(U_t^n(z))1_C(R^n_\theta(t))$ and $\upsilon^{\mathcal{A}_n^t(z)}\times \delta_{R^n_\theta(t)}(B)=\upsilon^{\mathcal{A}_n^t(z)}(A)1_C(R^n_\theta(t))$. Observe that by the definition of $\upsilon^{\mathcal{A}_n^t(z)}$, we have 
$$\upsilon^{\mathcal{A}_n^t(z)}(A)=\E_\upsilon(1_A\circ U_t^n|\mathcal{A}_n^t)(z).$$

Let $f_n(z)=\E_\upsilon(1_A\circ U_t^n|\mathcal{A}_n^t)(z)1_C(R^n_\theta(t))-1_A(U_t^n(z))1_C(R^n_\theta(t))$.  Note that $f_n$ is bounded uniformly in $n$. We only need to prove that $\lim_{N\to\infty}\frac{1}{N}\sum_{n=0}^{N-1}f_n(z)=0$ for  $\upsilon$-a.e. $z$.
For this, it is sufficient to show that for certain $k'\ge1$ and each $p=0,\ldots,k'-1$ we have
$\lim_{N\to\infty}\frac{1}{N}\sum_{n=0}^{N-1}f_{nk'+p}(z)=0$ for  $\upsilon\textrm{-a.e. }z$.

Now, for each $n\ge1$, we have $\E_\upsilon(f_n|\mathcal{A}_n^t)=0$. By the definition of the partitions $\{\mathcal{A}_n^{t}\}_n$ (see \eqref{eq: definition partition A-n-t 1}), it is clear that there exists $k''\in \N$ such that for all $s,s'\in [0,1)$ and all $n\ge1$, $\mathcal{A}_{n+k''}^{s}$ refines $\mathcal{A}_{n}^{s'}$. Because of this and since $A\in \mathcal{A}_k^{t'}$, the map $1_A\circ U_t^n$ is $\mathcal{A}_{n+k'}^t$-measurable for $k'=k+k''$.  Thus $\{f_{nk'+p}\}_n$ is a sequence of bounded martingale differences for the filtration $\{\mathcal{A}_{nk'+p}^t\}_n$, from which we deduce that their Ces\`aro averages converge to 0 for $\upsilon$-a.e. $z$, see \cite[Theorem 3 in Chapter VII.9]{Feller1971book}.

\end{proof}

\begin{lemma}\label{lemma invariance non-concentration}
For $Q_{1,3}$-a.e. $(\mu,t)$ and $\mu$-a.e. $x$ with $(\mu,x,t)\in \mathcal{E}_\gamma$, we have: for any $k\ge 1$ and each $B\in \mathcal{B}_k$,
\begin{equation}\label{eq lemma invariance non-concentration 0}
\limsup_{N\to\infty}\frac{1}{N}\sum_{n=0}^{N-1}(\pi_t\mu)^{\mathcal{A}_n^t(\pi_t(x))}\times \delta_{R^n_\theta(t)}((\partial B)^{(\epsilon)})=o(1) \ \textrm{  as } \epsilon\to 0,
\end{equation}
where $E^{(\epsilon)}$ denotes the $\epsilon$-neighborhood of a set $E$.
\end{lemma}
\begin{proof}
By the global adaptedness of $Q$, we only need to show \eqref{eq lemma invariance non-concentration 0} for $Q$-a.e. $(\mu,x,t)\in \mathcal{E}_\gamma$.

Fix $k\ge1$ and let $B\in \mathcal{B}_k$.
Recall that $B=A\times C$ with  $C\in \mathcal{C}_k$ and $A\in \mathcal{A}_k^t$ for some $t\in C$. Observe that we have  $(\partial B)^{(\epsilon)}\subset \left(K\times (\partial C)^{(\epsilon)}\right)\bigcup \left((\partial A)^{(\epsilon)}\times [0,1)\right).$ Thus it is sufficient to show that for $Q$-a.e. $(\mu,x,t)\in \mathcal{E}_\gamma$,
\begin{equation}\label{eq: lemma invariance non-concentration 1}
\limsup_{N\to\infty}\frac{1}{N}\sum_{n=0}^{N-1} \delta_{R^n_\theta(t)}((\partial C)^{(\epsilon)})=o(1) \ \textrm{  as } \epsilon\to 0
\end{equation}
and 
\begin{equation}\label{eq: lemma invariance non-concentration 2}
\limsup_{N\to\infty}\frac{1}{N}\sum_{n=0}^{N-1}(\pi_t\mu)^{\mathcal{A}_n^t(z)}((\partial A)^{(\epsilon)})=o(1) \ \textrm{  as } \epsilon\to 0.
\end{equation}
The statement \eqref{eq: lemma invariance non-concentration 1} is clearly true. Actually, since $\theta$ is irrational, for any $t\in [0,1)$, the limsup in \eqref{eq: lemma invariance non-concentration 1} is a limit and it is bounded by the Lebesgue measure of  $(\partial C)^{(\epsilon)}$ which is $o(1)$ when $\epsilon\to 0$.

Now, let us prove \eqref{eq: lemma invariance non-concentration 2}. 
In the proof of Lemma  \ref{lemma generic measure convergence 1}, we have seen that for $Q$-a.e. $(\mu,x,t)\in \mathcal{E}_\gamma$,   $(\mu,x)$ generates an ergodic CP-distribution $Q^{(\mu,x,t)}_{1,2}$ with positive dimension.
It follows from Proposition \ref{prop non-concentration and entropy}, \eqref{eq: prop non-concentration 1} (and Remark \ref{remark prop non-concentration and entropy}) that for any $\epsilon>0$, there exists $n_0(\epsilon)\in \N$ such that 
\begin{equation}
\liminf_{N\to\infty}\frac{1}{N}\sharp \left\{1\le k\le N : \max_{u\in \Lambda^{n_0(\epsilon)}}\mu^{[x_1^k]}([u])\le \epsilon\right\}>1-\epsilon.
\end{equation}
Now, recalling $\pi_{R_\theta^k(t)}(\mu^{[x_1^k]})=(\pi_t\mu)^{\mathcal{A}_k^t(\pi_t(x))}$ and  using part (3) of Lemma \ref{lemma basic propeties 1}, we deduce that for any $\epsilon>0$ there exists $\delta(\epsilon)>0$  such that
\begin{equation}\label{eq: lemma invariance non-concentration 3}
\liminf_{N\to\infty}\frac{1}{N}\sharp \left\{1\le k\le N : \sup_{y\in K}(\pi_t\mu)^{\mathcal{A}_k^t(\pi_t(x))}(B(y,\delta(\epsilon)))\le \epsilon
\right\}>1-\epsilon.
\end{equation} 
By definition, all elements in $\mathcal{A}_n^t$ have uniformly bounded eccentricities\footnote{The eccentricity of a rectangle is the ratio of the lengths of the longest and shortest side. Here we are actually referring to the eccentricity of the convex hull of $\mathcal{A}_n^t$ but not itself, since $\mathcal{A}_n^t$ is in general a Cantor set.} (less than $1/\beta$). On the other hand, the measure $(\pi_t\mu)^{\mathcal{A}_n^t(\pi_t(x))}$ is  supported on some slice of $K$ with slope between $1$ and $1/\beta$. Hence there exists an absolute constant depending only on $\beta$ such that for any $A\in\mathcal{A}_n^t$, the intersection of the support of $(\pi_t\mu)^{\mathcal{A}_n^t(\pi_t(x))}$ with $(\partial A)^{(\epsilon)}$ is included in two balls of diameter less than $\epsilon$ times this constant. Combining this fact with \eqref{eq: lemma invariance non-concentration 3}, we get \eqref{eq: lemma invariance non-concentration 2}.
\end{proof}

The following lemma says that the measures $\eta_N(z,t)$ and the Ces\`aro averages of $(\pi_t\mu)^{\mathcal{A}_n^t(\pi_t(x))} \times \delta_{R_\theta^n(t)}$ are asymptotically the same for typical $(\mu,x,t)$ in $\mathcal{E}_\gamma$.

\begin{lemma}\label{lemma invariance 2}
For $Q_{1,3}$-a.e. $(\mu,t)$ and $\mu$-a.e. $x$ with $(\mu,x,t)\in \mathcal{E}_\gamma$, we have
$$\eta_N(\pi_t(x),t) \to \nu^{(\mu,x,t)} \ \textrm{  as } N\to\infty$$
in the weak-*  topology.
\end{lemma}

\begin{proof}
By the definition of $\{\mathcal{B}_n\}_n$, it is clear that $\max_{B\in \mathcal{B}_n}{\rm diam}(B)\to 0$ as $n\to\infty$. So the partitions $\{\mathcal{B}_n\}_n$ generate the Borel $\sigma$-algebra of $K\times [0,1)$. Now by this fact and Lemma \ref{lemma invariance 1},  it is well known that for proving Lemma \ref{lemma invariance 2} we only need to show the following: for $Q_{1,3}$-a.e. $(\mu,t)$ and $\mu$-a.e. $x$ with $(\mu,x,t)\in \mathcal{E}_\gamma$, whenever $\eta_{N_k}(\pi_t(x),t) \to \upsilon$ along some $N_k\to\infty$, then $\upsilon(\partial B)=0$ for each $B\in \mathcal{B}_n$ and all $n\ge1$. For this, we use Lemma \ref{lemma invariance non-concentration}. Fix any $n_0\ge1$ and $B\in \mathcal{B}_{n_0}$. For any $\epsilon>0$, let $f_\epsilon\in C(K\times [0,1))$ be such that $1_{\partial B}\le f_\epsilon\le 1_{(\partial B)^{(\epsilon)}}$. Since $\max_{B\in \mathcal{B}_k}{\rm diam}(B)\to 0$ as $k\to\infty$, for $n$ large enough we can find a finite family $\{B_i\}\subset \mathcal{B}_n$ such that $(\partial B)^{(\epsilon)}\subset \cup_iB_i\subset (\partial B)^{(2\epsilon)}$. 
Now if $\eta_{N_k}(\pi_t(x),t) \to\upsilon$, then by Lemma \ref{lemma invariance 1} and Lemma \ref{lemma invariance non-concentration}, we have
\begin{eqnarray*}
\int f_\epsilon d\upsilon=\lim_{k\to\infty}\frac{1}{N_k}\sum_{n=0}^{N_k-1}f_\epsilon(U^n(\pi_t(x),t))  & \le & \limsup_{N\to\infty}\frac{1}{N}\sum_{n=0}^{N-1}(\pi_t\mu)^{\mathcal{A}_n^t(\pi_t(x))}\times \delta_{R^n_\theta(t)}(\cup_iB_i) \\
 & \le & \limsup_{N\to\infty}\frac{1}{N}\sum_{n=0}^{N-1}(\pi_t\mu)^{\mathcal{A}_n^t(\pi_t(x))}\times \delta_{R^n_\theta(t)}((\partial B)^{(2\epsilon)}) \\
 &=& o(1) \ \textrm{  as } \epsilon\to 0.
\end{eqnarray*}
This implies that $\upsilon(\partial B)=0$.
\end{proof}

We are now ready to prove Proposition \ref{proposition invariance 1}.
\begin{proof}[Proof of Proposition \ref{proposition invariance 1}]

By Lemma \ref{lemma invariance 2}, for $Q_{1,3}$-a.e. $(\mu,t)$ and $\mu$-a.e. $x$ with $(\mu,x,t)\in \mathcal{E}_\gamma$, $\nu^{(\mu,x,t)}$ is a measure according to which certain orbit $\{U^n(z,t)\}_n$ equidistributes.  Thus for proving the $U$-invariance of $\nu^{(\mu,x,t)}$, we only need to show that it gives zero measure to the set of discontinuities of $U$. This is an immediate consequence of the fact that  the  discontinuities of $U$ are contained in the set $\bigcup_{B\in \mathcal{B}_1}\partial B$, since in the proof of Lemma \ref{lemma invariance 2} we have shown that $\nu^{(\mu,x,t)}$ gives zero measure to this set. 
\end{proof}

\smallskip

\subsection{Further properties of a $U$-invariant measure $\nu_\infty$}\label{subsection Properties and entropy of the measure}
From now on, let us fix an element $(\mu_0,x_0,t_0)\in \mathcal{E}_\gamma$ such that $Q^{(\mu_0,x_0,t_0)}_{1,2}$ is an  ergodic CP-distribution with dimension $\ge \gamma$ and the measure $$\nu_\infty:=\nu^{(\mu_0,x_0,t_0)}=\int \pi_s\mu\times \delta_{t} \ dQ^{(\mu_0,x_0,t_0)}_{1,3}(\mu,t)$$
is $U$-invariant.

Applying Proposition \ref{prop non-concentration and entropy} to the ergodic CP-distribution $Q^{(\mu_0,x_0,t_0)}_{1,2}$ we get: for any $\epsilon>0$, there exists $n_0(\epsilon)\in \N$ such that for $Q^{(\mu_0,x_0,t_0)}_{1}$-a.e. $\mu$ and $\mu$-a.e. $x$, 
\begin{equation}
\begin{split}
\liminf_{N\to\infty}\frac{1}{N}\sharp \bigg\{1\le k\le N : \max_{u\in \Lambda^{n_0(\epsilon)}}\mu^{[x_1^k]}([u])\le  \epsilon \  \textrm{ and } &\\
 H(\mu^{[x_1^k]},\mathcal{F}_n)\ge n(\gamma\log \alpha^{-1}-\epsilon)\bigg\}>&1-2\epsilon \ \textrm{ for all } n\ge n_0(\epsilon).
\end{split}
\end{equation}
Here we use $Q^{(\mu_0,x_0,t_0)}_{1}$ to denote the measure marginal of $Q^{(\mu_0,x_0,t_0)}_{1,2}$.
Now, using part (3) of Lemma \ref{lemma basic propeties 1}, we deduce that for any $\epsilon>0$ there exists $\delta(\epsilon)>0$ and $n_1(\epsilon)\in \N$ such that for $Q^{(\mu_0,x_0,t_0)}_{1}$-a.e. $\mu$ and $\pi_t\mu$-a.e. $z$,
\begin{equation}\label{eq: property nu-infty non-concentration and entropy 1}
\begin{split}
\liminf_{N\to\infty}\frac{1}{N}\sharp \bigg\{1\le k\le N : \sup_{y\in K}(\pi_t\mu)^{\mathcal{A}_k^t(z)}(B(y,\delta(\epsilon)))\le \epsilon \ & \textrm{ and } \\
 H((\pi_t\mu)^{\mathcal{A}_k^t(z)},\mathcal{D}_n)  \ge n(\gamma\log 2-2\epsilon)\bigg\}>&1-2\epsilon \ \textrm{ for all } n\ge n_1(\epsilon).
\end{split}
\end{equation}
In particular, the above property holds also for $Q^{(\mu_0,x_0,t_0)}_{1,3}$-a.e. $(\mu,t)$ and $\pi_t\mu$-a.e. $z$.
On the other hand, since the measure $\nu_\infty$ has the form $\int \pi_t\mu\times \delta_{t} \ dQ_{1,3}^{(\mu_0,x_0,t_0)}(\mu,t)$, selecting a pair $(z,t)$ according to $\nu_\infty$ can be done by first selecting a pair $(\mu,t)$ according to $Q_{1,3}^{(\mu_0,x_0,t_0)}$ and then selecting a point $z$ according to $\pi_t\mu$.  

It follows from the above discussions that  we have 
\begin{proposition}\label{proposition, property slice measure 1}
The measure $\nu_\infty$ satisfies the following property:
\begin{equation}\label{eq: property slice measure 1}
\begin{aligned}
&\textit{For any $\epsilon>0$, there are $\delta(\epsilon)>0$ and $n_1(\epsilon)$ such that for $\nu_\infty$-a.e. $(z,t)$, }\\
& \textit{we  can find $\mu\in\mathcal{P}(X)$ such that $\pi_t\mu\in \mathcal{P}(\ell\cap K)$ for some line $\ell$ with  }\\
& \textit{slope $\beta^{-t}$ and \eqref{eq: property nu-infty non-concentration and entropy 1} holds for $\pi_t\mu$ and $z$.}
\end{aligned}
\end{equation}
In particular, almost every ergodic component of $\nu_\infty$ still satisfies the property \eqref{eq: property slice measure 1}.
\end{proposition}

In the rest of this paper, we choose an ergodic component of $\nu_\infty$ which  satisfies the property \eqref{eq: property slice measure 1} and  still denote it by $\nu_\infty$. We have
thus proved the following:
\begin{theorem}\label{theorem property slice measure and entropy 1}
There exists a $U$-invariant ergodic measure $\nu_\infty$ which  satisfies the property \eqref{eq: property slice measure 1}.
\end{theorem}

\medskip

\section{An ergodic theoretic result}\label{section ergodic results}

This section is devoted to the proof of the following theorem in ergodic theory.
Recall that a sequence $\{x_k\}_{k\in \N}\in [0,1)$ is called {\em uniformly distributed} (UD) if for any sub-interval $J$ of $[0,1)$ we have
$\lim_{N\to\infty} N^{-1}\sharp\left\{0\le k\le N-1:  x_k\in J\right\}=\mathcal{L}(J).$

\begin{theorem}\label{proposition ergodic results}
Let $(X,T,\mu)$ be an ergodic dynamical system. Let $\mathcal{A}$ be a generator with finite cardinality and let $\{\mathcal{A}_n\}_n$ be the filtration generated by $\mathcal{A}$ with respect to $T$ (see Section \ref{subsubsection Measure  entropy1}). Suppose that $\mu(\partial A)=0$ for each $A\in \mathcal{A}_n$ and all $n\ge1$. Let $\xi$ be an irrational number. For any $\epsilon>0$, there exists $n_2=n_2(\epsilon)\in \N$ such that for each $n\ge n_2$ we can find a disjoint family $\{C_i\}_{i=1}^{N(n,\epsilon)}$ of measurable subsets  $C_i\subset X$ satisfying the following properties: \begin{itemize}
\item[(1)] We have $\mu\left(\bigcup_iC_i\right)\ge 1-\epsilon$.

\item[(2)] For each $1\le i\le N(n,\epsilon)$, we have $\sharp\left\{A\in \mathcal{A}_n: C_i\bigcap A\neq \emptyset\right\}\le e^{ n \epsilon}$.

\item[(3)]  There exists another disjoint family $\{\widetilde{C}_i\}_{i=1}^{N(n,\epsilon)}$ of measurable subsets  $\widetilde{C}_i\subset X$ such that for each $1\le i\le N(n,\epsilon)$, we have $C_i\subset\widetilde{C}_i$ and  $\mu(C_i)\ge(1-\epsilon)\mu(\widetilde{C}_i)$,  and moreover, for $\mu$-a.e. $x$  the sequence $$\left\{R_\xi^k(0)\in [0,1): k\in \N \textrm{ and  } T^k(x)\in \widetilde{C}_i\right\}$$ is UD. Here $R_\xi$ is the irrational rotation map defined by $R_\xi(t)=t-\xi \mod 1$. 

\end{itemize}
\end{theorem}
\begin{remark}
The conclusion of the theorem holds without the condition that the generator $\mathcal{A}$ has finite cardinality, but we will not use this fact. Assuming the condition on $\mathcal{A}$ will make the proof shorter.
\end{remark}

We will use Sinai's  factor theorem in the proof of Theorem \ref{proposition ergodic results}.

\begin{theorem}[Sinai's factor theorem]
Let $(X, T,\mu) $ be an ergodic dynamical system. Then any Bernoulli shift $(\Sigma^\N,\sigma,\nu)$ with $h(\nu,\sigma)\le h(\mu,T)$ is a factor of $(X, T,\mu) $.
\end{theorem}

For definitions of Bernoulli shift and factor, see Subsection \ref{preliminary on DS 1}.
In the above version of Sinai's factor theorem, we include the case when $h(\mu,T)=0$ -- in this case the theorem is obviously true since every Bernoulli shift with zero entropy is a trivial one-point system (the product measure $\nu$ is a Dirac measure at a fixed point) which is trivially a factor of $(X, T,\mu) $.
The original version of Sinai's factor theorem \cite{Sinai62,Sinai64} was stated for invertible systems, but it also implicitly applies to non-invertible ones (for the proof see also \cite{OW75}).

\begin{quote}
{\em For the rest of this section, we fix an ergodic dynamical system $(X, T,\mu) $ satisfying the hypothesis of Theorem \ref{proposition ergodic results} and let $(\Sigma^\N,\sigma,\nu)$ be a Bernoulli shift with $h(\nu,\sigma)= h(\mu,T)$.} 
\end{quote}

It follows from Sinai's factor theorem that there exists a factor map $\pi:X\to \Sigma^\N$ such that $$\pi\circ T=\sigma\circ \pi \ \textrm{ and }\ \nu=\pi \mu.$$
By Rohlin's disintegration theorem, there exists a system of conditional measures $(\mu_y)_{y\in \Sigma^\N}$ of $\mu$ with respect to $\pi$ satisfying the following properties:
{\em \begin{itemize}
\item[(1)] For $\nu$-a.e. $y$, $\mu_y$ is a Borel probability measure supported on $\pi^{-1}(y)$.
\item[(2)] For every $\mu$-measurable $B\subset X$, the map $y\mapsto \mu_y(B)$ is $\nu$-measurable and 
$$\mu(B)=\int_{\Sigma^\N} \mu_y(B)d\nu(y).$$
\item[(3)] Moreover for $\nu$-a.e. $y$, the measure $\mu_y$ can be obtained as the weak-* limit of $\lim_{r\to 0}\mu_{\pi^{-1}(B(y,r))}$ where $\mu_{\pi^{-1}(B(y,r))}$ is defined by
$$\mu_{\pi^{-1}(B(y,r))}(A)=\frac{\mu\left(\pi^{-1}(B(y,r))\bigcap A\right)}{\mu\left(\pi^{-1}(B(y,r))\right)}.$$
\end{itemize} }
For a proof of the above version of Rohlin's disintegration theorem, see \cite{Simmons2012}.

 The proof of Theorem \ref{proposition ergodic results} relies on two lemmas.  Recall that $\{\mathcal{A}_n\}_n$ is the filtration associated to the generator $\mathcal{A}$ and, for $x\in X$, $\mathcal{A}_n(x)$ is the unique element of  $\mathcal{A}_n$ containing $x$.
\begin{lemma}\label{lemma condition measure egorov 1}
Suppose that $\mu$ satisfies the hypothesis of Theorem \ref{proposition ergodic results}. Let $\nu$ and $(\mu_y)_{y\in \Sigma^\N}$ be as above.
For any $\delta>0$, we have:
\begin{itemize}
\item[(i)] There exist a measurable set $A_\delta\subset X$ with $\mu(A_\delta)>1-\delta$ and $n'  \in \N$ such that for each $x\in A_\delta $,
\begin{equation}\label{eq: lemma condition measure egorov 1}
\mu_{\pi(x)}(\mathcal{A}_n(x))\ge e^{-n\delta} \ \textrm{ for all } n\ge n'.
\end{equation}
\item[(ii)] For any $n\ge 1$, there exist a measurable set $B_\delta^n\subset \Sigma^\N$ with $\nu(B_\delta^n)>1-\delta$ and $r=r(\delta,n)>0$ such that for each $y\in B_\delta^n$ and each $A\in \mathcal{A}_n$ we have  
\begin{equation}\label{eq: lemma condition measure egorov 2}
\frac{\mu\left(\pi^{-1}(B(y,r))\bigcap A\right)}{\mu\left(\pi^{-1}(B(y,r))\right)}\ge (1-\delta)\mu_{y}(A).
\end{equation}
\end{itemize}
\end{lemma}

\begin{proof}
{\rm (i)} Since  $(\Sigma^\N,\sigma,\nu)$ is a factor of $(X, T,\mu) $ with  $h(\nu,\sigma)= h(\mu,T)$, it follows from the conditional Shannon-McMillan-Breiman Theorem \cite[Theorem 3.3.7]{Downarowicz2011book} that for $\mu$-a.e. $x$, 
$$\lim_{n\to\infty}\frac{\log \mu_{\pi(x)}(\mathcal{A}_n(x))}{-n}=0.$$
By Egorov's theorem, there exist a measurable set $A_\delta\subset X$ with $\mu(A_\delta)>1-\delta$ and $n'\in \N$ such that for each $x\in A_\delta $, $$\frac{\log \mu_{\pi(x)}(\mathcal{A}_n(x))}{-n}\le \delta \ \textrm{ for all } n\ge n'.$$
This is exactly \eqref{eq: lemma condition measure egorov 1}.

{\rm (ii)} Fix any $n\ge 1$. By hypothesis, $\mu(\partial A)=0$ for all $A\in \mathcal{A}_n$. The same holds for $\mu_y$ for $\nu$-a.e. $y$.  Recall that by Rohlin's disintegration theorem,  for $\nu$-a.e. $y$, $\mu_y$ is the weak-* limit of $\mu_{\pi^{-1}(B(y,r))}$ as $r\to0$. Thus, by Portmanteau's theorem, we deduce that for $\nu$-a.e. $y$ and for all $A\in \mathcal{A}_n$,
$$\lim_{r\to\infty}\frac{\mu\left(\pi^{-1}(B(y,r))\bigcap A\right)}{\mu\left(\pi^{-1}(B(y,r))\right)}=\mu_y(A).$$
We can then again apply Egorov's theorem to obtain a measurable set $B_\delta^n\subset \Sigma^\N$ with $\nu(B_\delta^n)>1-\delta$ and $r=r(\delta,n)>0$ such that for each $y\in B_\delta^n$ and each $A\in \mathcal{A}_n$ we have  
\eqref{eq: lemma condition measure egorov 2}.
\end{proof}

The following result is an easy consequence of the mixing property of the  Bernoulli shift $(\Sigma^\N,\sigma,\nu)$.
\begin{lemma}\label{lemma mixing UD 1}
For any measurable set $B\subset \Sigma^{\N}$ with $\nu(B)>0$, the sequence $$\left\{R_\xi^k(0): k\in \N \ \textrm{ and } T^k(x)\in \pi^{-1}(B)\right\}$$ is UD for $\mu$-a.e. $x\in X$.
\end{lemma}
\begin{proof}
Since the Bernoulli shift $(\Sigma^\N,\sigma,\nu)$ is weak-mixing, for any irrational rotation system $([0,1),R_\xi,\mathcal{L})$, the product system $(\Sigma^\N\times [0,1),\sigma\times R_\xi,\nu\times \mathcal{L})$ is ergodic (see Subsection \ref{subsection preliminary on DS 2}). We claim that if $B\subset \Sigma^{\N}$ is measurable with $\nu(B)>0$, then  the set 
$$\left\{R_\xi^k(0): k\in \N \ \textrm{ and } \sigma^k(y)\in B\right\}$$
is UD for $\nu$-a.e. $y\in \Sigma^{\N}$. To see this, note that by the ergodic theorem, for $\nu$-a.e. $y$ and $ \mathcal{L}$-a.e. $t$, the sequence $\{x_n(y,t)\}_n:=\{R_\xi^k(t): k\in \N \ \textrm{ and } \sigma^k(y)\in B\}$ satisfies
$\lim_{N\to\infty}N^{-1}\sharp\{1\le n\le N: x_n(y,t)\in J\}=\mathcal{L}(J)$ for each dyadic interval $J\in \mathcal{D}_k([0,1))$, $k\ge1$. This clearly implies that the sequence $\{x_n(y,t)\}_n$ is UD. Since $R_\xi^k(t)=R_\xi^k(0)+t$ in $[0,1)$, we deduce that $\{x_n(y,0)\}_n$ is UD for $\nu$-a.e. $y$, as claimed.

On the other hand, since $(\Sigma^\N,\sigma,\nu)$ is a factor of $(X,T,\mu)$ with factor map $\pi$, we have for $\mu$-a.e. $x\in X$,
$$\{k\in \N:T^k(x)\in \pi^{-1}(B)\}=\{k\in \N:\sigma^k(\pi(x))\in B\}.$$
Combining this with the above claim, we get the desired result.
\end{proof}

\medskip

\begin{proof}[Proof of Theorem \ref{proposition ergodic results}]
Fix $\epsilon>0$. Let $\delta>0$ be a small constant which we will choose later. Let $A_{\delta}$ and $n'=:n_2$ be the set and the number provided by Lemma \ref{lemma condition measure egorov 1}, (i). Then we have
$$\int_{\Sigma^\N}\mu_y(A_{\delta})d\nu(y)=\mu(A_{\delta})>1-\delta.$$
From this, we deduce that there exists  $\delta_1>0$, with $\delta_1=o(1)$ when $\delta\to 0$, so that the following holds: we can find a measurable set $B_1\subset \Sigma^\N$ with $\nu(B_1)>1-\delta_1$ such that for each $y\in B_1$, we have $\mu_y(A_\delta)>1-\delta_1$. For instance, we can take $\delta_1=\sqrt{\delta}$.  

Fix any $n\ge n_2$. Let $B_\delta^n$ and $r$ be the measurable set and the number provided by Lemma \ref{lemma condition measure egorov 1}, (ii). Note that we have $\nu(B_\delta^n)>1-\delta$. Let $B_2=B_1\cap B_\delta^n$. Then we have $\nu(B_2)>1-\delta-\delta_1$. For each $y\in B_2$, let
$$E(y,n)=\left\{A\in \mathcal{A}_n: \pi^{-1}(y)\cap A_\delta\cap A\neq \emptyset\right\}.$$
By the definition of $A_\delta$, if $x\in A_\delta$, then $\mu_{\pi(x)}(\mathcal{A}_n(x))\ge e^{-n\delta}$. It follows that for each $A\in E(y,n)$ we have $\mu_y(A)\ge e^{-n\delta}$. Since $\mu_y$ is a probability measure, we deduce that $\sharp (E(y,n))\le e^{n\delta}$ for each $y\in B_2$.  

Now, let us consider the following collection of balls of  $\Sigma^\N$:
$$\left\{B(y,r)\subset \Sigma^\N:y\in B_2 \ \textrm{ and } \nu(B(y,r))>0\right\}.$$
Since we use an ultra-metric in $\Sigma^\N$, the above collection is actually finite. Let us enumerate its elements  by $\{B_i\}_{i=1}^{N(n)}$. Note that $B_i$'s are disjoint balls. For each $1\le i\le N(n)$, let us define 
$$\widetilde{C}_i=\pi^{-1}(B_i) \ \ \textrm{ and }\ \    C_i=\pi^{-1}(B_i)\bigcap\left(\bigcup_{A\in E(y,n)}A\right), $$
where $y$ is some point in $B_2$ such that $B(y,r)=B_i$.  
Now we can make our choice of $\delta$. In the following we fix $\delta$ small enough such that 
$$\delta\le \epsilon \ \textrm{ and }\ (1-\delta-\delta_1)(1-\delta)(1-\delta_1)\ge 1-\epsilon.$$
Let $N(n,\epsilon):=N(n)$.
We claim that the families $\{C_i\}_{i=1}^{N(n,\epsilon)}$ and $\{\widetilde{C}_i\}_{i=1}^{N(n,\epsilon)}$ satisfy the properties (1), (2) and (3) in Theorem \ref{proposition ergodic results}.

We first verify the property (2). We have seen that $\sharp (E(y,n))\le e^{n\delta}$ for each $y\in B_2$.  By the definition of $C_i$ and the assumption $\delta\le \epsilon$, this clearly implies the property (2). 

Now, we verify the properties (1) and (3).
Observe that $\mathcal{A}_n$ is a partition of $X$, thus by definition of $E(y,n)$ we have  for $y\in B_2$,
$$\pi^{-1}(y)\cap A_\delta \subset \bigcup_{A\in E(y,n)}A.$$
Note that by the choice of $B_1$, we have
$$\mu_y\left(\pi^{-1}(y)\cap A_\delta\right)=\mu_y( A_\delta)> 1-\delta_1$$
for each $y\in B_1$.
From these two facts, we deduce that if $y\in B_2\subset B_1$, then 
\begin{equation}\label{eq: proof prop erg results 1}
\mu_y\left(\bigcup_{A\in E(y,n)}A\right)\ge 1-\delta_1.
\end{equation}
On the other hand, recall  that each $y\in B_\delta^n$ satisfies \eqref{eq: lemma condition measure egorov 2} for all $A\in \mathcal{A}_n$. Using this and the fact $B_2\subset B_\delta^n$, we deduce from   from  \eqref{eq: proof prop erg results 1} that for each $y\in B_2$, we have
\begin{equation*}
\mu\left(\pi^{-1}(B(y,r))\bigcap\left(\bigcup_{A\in E(y,n)}A\right)\right)\ge (1-\delta)(1-\delta_1)\mu\left(\pi^{-1}(B(y,r))\right).
\end{equation*}
Combining this with the definitions of $C_i$ and $\widetilde{C}_i$ and the choice of $\delta$, we get 
$$\mu(C_i)\ge (1-\delta)(1-\delta_1)\mu(\widetilde{C}_i)\ge (1-\epsilon)\mu(\widetilde{C}_i)$$
for each $1\le i\le N(n,\epsilon)$.
Note also that $$\mu\left(\cup_i\widetilde{C}_i\right)=\mu\left(\cup_i\pi^{-1}(B_i)\right)=\nu\left(\cup_i B_i\right)\ge \nu(B_2)\ge 1-\delta-\delta_1.$$ 
Thus again by the choice of $\delta$, we obtain
$$\mu\left(\cup_i C_i\right)\ge (1-\delta)(1-\delta_1)\mu\left(\cup_i\widetilde{C}_i\right)\ge (1-\delta-\delta_1)(1-\delta)(1-\delta_1)\ge 1-\epsilon.$$
It remains to show that the sequence $$\left\{R_\xi^k(0)\in [0,1): k\in \N \textrm{ and  } T^k(x)\in \widetilde{C}_i\right\}$$
is UD on $[0,1)$. This is implied by Lemma \ref{lemma mixing UD 1}.
\end{proof}

\medskip

\section{Proof of Theorem \ref{main theorem 1}}\label{section Proof of Theorem main 1}

The following result is essential for proving Theorem \ref{main theorem 1}. It is a consequence of the property \eqref{eq: property slice measure 1} of $\nu_\infty$ and 
an application of Theorem \ref{proposition ergodic results}  to the system $(K\times [0,1), U, \nu_\infty)$.
Recall that $\Pi_1$ is the projection from $K\times [0,1)$ to $K$ and $N_{2^{-n}}(A)$ denotes the number of $n$-level dyadic cubes intersecting a set $A$. 

\begin{proposition}\label{proposition good geometric picture 1}
For any $\epsilon>0$, there exist $r_0=r_0(\epsilon)>0$ and $n_3=n_3(\epsilon)\in \N$  such that for each $n\ge n_3$ the following is true: for $\nu_\infty$-a.e. $(z,t)$ we can find a measure $\nu\in \mathcal{P}(K)$, a measurable set $D\subset K\times [0,1)$ and a subset $\mathcal{N}\subset \N$ satisfying the properties:
\begin{itemize}
\item[(1)] The measure  $\nu\in \mathcal{P}(\ell\cap K)$ for some line $\ell$ with slope $\beta^{-t}$.
\item[(2)]    $n^{-1}\log N_{2^{-n}}(\Pi_1(D))\le \epsilon$.
\item[(3)] For each $k\in \mathcal{N} $, $U^k(z,t)\in D$.
\item[(4)] $\mathcal{L}\left(\overline{\{R^k_\theta(t): k\in \mathcal{N}\}}\right)\ge 1-\epsilon$,  where $\mathcal{L}$ denotes the normalized Lebesgue measure on $[0,1)$ (i.e., $\mathcal{L}([0,1))=1$). 
\item[(5)] For each $k\in \mathcal{N} $, 
$$\inf_{y\in K} \frac{1}{n\log 2}H\left(\nu^{\mathcal{A}_k^t(z)}|_{B(y,r_0)^{c}},\mathcal{D}_n\right)\ge \gamma-\epsilon^{\frac{1}{2}}.$$
\end{itemize}
\end{proposition}      
Recall that $\nu^{\mathcal{A}_k^t(z)}$  is defined as \eqref{eq: def magnification measure 1}, and it is supported on some slice $\ell'\cap K$ with slope $\beta^{-R_\theta^k(t)}$.  Recall also that $\eta|_{E}$ denotes the restriction of a measure $\eta$ on $E$; see Section \ref{subsection Partition and entropy 1} for the definition of entropy.

For the proof of Proposition \ref{proposition good geometric picture 1}, we need two elementary lemmas.
For $F_1\subset F_2\subset \N$, we define the {\em upper density} of $F_1$ in $F_2$, denoted $\overline{d}(F_1,F_2)$, as 
$$\overline{d}(F_1,F_2)=\limsup_{N\to\infty}\frac{\sharp\{F_1\cap [0,N-1]\}}{\sharp\{F_2\cap [0,N-1]\}}.$$
Similarly, we define the {\em lower density} $\underline{d}(F_1, F_2)$ of $F_1$ in $F_2$. If $\overline{d}(F_1,F_2)=\underline{d}(F_1,F_2)$, then we say the density of $F_1$ in $F_2$ exists and denote it by $d(F_1,F_2)$. 
\begin{lemma}\label{lemma positive density positive measure 1}
Let $\{x_k\}_{k\in\N}\subset [0,1)$ be a sequence which is UD. Suppose that $F\subset\N$, then 
$$\mathcal{L}\left(\overline{\left\{x_k:k\in F\right\}}\right)\ge \overline{d}(F,\N).$$
\end{lemma}
\begin{proof}
Let $E=\overline{\left\{x_k:k\in F\right\}}$. If $\mathcal{L}( E^c)>0$, then for any $\epsilon>0$, we can find finitely many intervals $\{J_i\}_i\subset E^c$ such that $\mathcal{L}(\cup_iJ_i)>\mathcal{L}( E^c)-\epsilon$. Now since $\{x_k\}_{k\in \N}$ is UD, we have
\begin{eqnarray*}
\mathcal{L}(\cup_iJ_i) & = & \lim_{N\to\infty}N^{-1}\sharp\{1\le k\le N:x_k\in \cup_iJ_i\} \\
 & = & 1-\lim_{N\to\infty}N^{-1}\sharp\{1\le k\le N:x_k\notin \cup_iJ_i\}\le 1-\overline{d}(F).
\end{eqnarray*}
\end{proof}

\begin{lemma}\label{lemma non-concent and large ent consequence 1}
Let $\eta \in \mathcal{P}(\R^d)$ and $0<\delta<1$. If $\sup_{y\in \R^d}\eta(B(y,\delta))\le \epsilon$, then for $n\in \N$ with $2^{-n}\le \delta $, we have 
$$\inf_{y\in \R^d}H(\eta|_{B(y,\delta)^{c}},\mathcal{D}_n)\ge H(\eta,\mathcal{D}_n)-C_1n\epsilon^{\frac{1}{2}}  $$
for some constant $C_1$ depending only on $d$.
\end{lemma}
\begin{proof}
We will use the elementary fact that if $\mu$ is a finite (not necessarily probability) measure on a metric space $X$, then for any finite partition $\mathcal{A}=\{A_i\}_{i=1}^k$ of $X$, we have
\begin{equation}\label{eq: lemma non-concent and large ent consequence, basic 1}
H(\mu,\mathcal{A})\le\sum_i\frac{\mu(X)}{k}\log \frac{k}{\mu(X)}= \mu(X)\log k+\mu(X)\log \frac{1}{\mu(X)}, 
\end{equation}
with equality only if $\mu(A_i)=\mu(X)/k$ for each $i$. 

Recall that $\mathcal{D}_n$ is the  collection of $n$-th level dyadic cubes of $\R^d$. Fix any $y_0\in \R^d$. Let $\mathcal{A}=\{w\in \mathcal{D}_n: w\cap B(y_0,\delta)\neq \emptyset\}$ and $E=\cup_{w\in \mathcal{A}}w$. Note that since
$2^{-n}\le \delta $, for some constant $C'=C'(d)$,  we have ${\rm diam}(E)\le C'\delta $ and $E$ can be covered by less than $C'$ balls of diameter $\delta$, thus $\eta(E)\le C'\epsilon$. Now to conclude the proof we only need to notice that 
$$H(\eta|_{B(y_0,\delta)^{c}},\mathcal{D}_n)\ge H(\eta,\mathcal{D}_n)-H(\eta|_E,\mathcal{A})$$
and by \eqref{eq: lemma non-concent and large ent consequence, basic 1},
$$H(\eta|_E,\mathcal{A})\le \eta(E)\log \sharp\mathcal{A}+\eta(E)\log \frac{1}{\eta(E)}\le C_1n\epsilon^{\frac{1}{2}}$$
for some constant $C_1=C_1(d)$.
\end{proof}

Now we are ready to prove Proposition \ref{proposition good geometric picture 1}.
\begin{proof}[Proof of Proposition \ref{proposition good geometric picture 1}.]
Fix any $\epsilon>0$.
Recall that by Theorem \ref{theorem property slice measure and entropy 1}, the measure $\nu_\infty$ is ergodic and satisfies the property \eqref{eq: property slice measure 1}. Let $r_0(\epsilon):=\delta(\epsilon),$ where $\delta(\epsilon)$ is the constant appearing in property \eqref{eq: property slice measure 1}.   

Recall that $\mathcal{B}_1$ is the partition of $K\times [0,1)$ defined in \eqref{eq: def of partition level 1}. Since $\mathcal{B}_1$ is a generator with finite cardinality and $\nu_\infty(\partial B)=0$ for each $B\in \mathcal{B}_n$, $n\ge 1$ (see the proof of Lemma \ref{lemma invariance 2}), we can apply Theorem \ref{proposition ergodic results} to the system $(K\times [0,1), U,\nu_\infty)$. Let $n_2(\epsilon )$ be the integer provided by Theorem \ref{proposition ergodic results}. Let $$n_3(\epsilon):=\max\{n_2(\epsilon ),n_2(\epsilon )\frac{\log\alpha^{-1}}{\log 2},n_1(\epsilon )\},$$ where  $n_1(\epsilon )$ is the integer appearing in  \eqref{eq: property slice measure 1}. 

We fix any $n\ge n_3(\epsilon)$.   Let $\widetilde{n}=\lfloor n\frac{\log 2}{\log\alpha^{-1}} \rfloor+1$, where $\lfloor x\rfloor$ denotes the integer part of $x$. By the choice of $ n_3(\epsilon)$, we have $\widetilde{n}\ge n_2(\epsilon ).$ Then by Theorem \ref{proposition ergodic results},  we can find a disjoint family $\{C_i\}_{i=1}^{N(\widetilde{n},\epsilon)}$ of measurable subsets  $C_i\subset K\times [0,1)$ satisfying the following properties: 
\begin{itemize}
\item[(i)] We have $\nu_\infty(\bigcup_iC_i)\ge 1-\epsilon$.

\item[(ii)] For $1\le i\le N(\widetilde{n},\epsilon)$, we have $\sharp\left\{E\in \mathcal{B}_{\widetilde{n}}: C_i\cap E\neq \emptyset\right\}\le e^{\epsilon \widetilde{n}}$.

\item[(iii)]  There exists another disjoint family $\{\widetilde{C}_i\}_{i=1}^{N(\widetilde{n},\epsilon)}$ of measurable subsets  $\widetilde{C}_i\subset K\times [0,1)$ such that for each $1\le i\le N(\widetilde{n},\epsilon)$, we have $C_i\subset\widetilde{C}_i$, $\nu_\infty(C_i)\ge(1-\epsilon)\nu_\infty(\widetilde{C}_i)$ and  for $\nu_\infty$-a.e. $(z,t)$  the sequence 
\begin{equation}\label{eq: proof proposition good geometric picture 1}
\left\{R_\theta^k(t)\in [0,1): k\in \N \textrm{ and  } U^k(z,t)\in \widetilde{C}_i\right\}
\end{equation}
  is UD. 
\end{itemize}
Now, it follows from the above property (iii) and the property \eqref{eq: property slice measure 1} that the following set 
$$
A' := \left\{ \begin{array}{ll}
(z,t) :&\textrm{the sequence \eqref{eq: proof proposition good geometric picture 1} is UD for each $1\le i\le N(\widetilde{n},\epsilon)$ and there}\\
        &\textrm{exists $\mu=\mu_{z,t}$ such that $\pi_t\mu\in \mathcal{P}(l\cap K)$ for some line $l$ with}\\
        &\textrm{slope $\beta^{-t}$ and \eqref{eq: property nu-infty non-concentration and entropy 1} holds for $\pi_t\mu$ and $z$.}
\end{array} \right\}
$$
has full $\nu_\infty$-measure.  For $1\le i\le N(\widetilde{n},\epsilon)$, let 
$$H(C_i,z,t)=\left\{k\in \N : U^k(z,t)\in C_i\right\}  \ \textrm{ and }\ H(\widetilde{C}_i,z,t)=\left\{k\in \N \textrm{ and  } U^k(z,t)\in \widetilde{C}_i\right\}.$$  
Let $A''$ be the set of $(z,t)$ such that for each $i$, 
$$d(H(C_i,z,t),\N)=\nu_\infty(C_i)\ \textrm{ and }\ d(H(\widetilde{C}_i,z,t),\N)=\nu_\infty(\widetilde{C}_i).$$ 
Recall that for a subset $F$ of $\N$,  $d(F,\N)$ denotes the density of $F$ in $\N$. 
By ergodicity of $\nu_\infty$, $A''$ also has full $\nu_\infty$-measure.
Let $A=A'\cap A''$. Then we still have $\nu_\infty(A)=1$.  

Now, let us pick any $(z,t)\in A$. In the following, we will find a measure $\nu\in \mathcal{P}(K)$, a measurable set $D\subset K\times [0,1)$ and a subset $\mathcal{N}\subset \N$ satisfying the properties (1)-(5) in the statement of Proposition \ref{proposition good geometric picture 1}.  

Note that since $A\subset A'$, $(z,t)\in A'$. It follows that  there exists $\mu=\mu_{z,t}$ such that $\pi_t\mu\in \mathcal{P}(\ell\cap K)$ for some line $\ell$ with slope $\beta^{-t}$ and \eqref{eq: property nu-infty non-concentration and entropy 1} holds for $\pi_t\mu$ and $z$. Let $$\nu=\pi_t\mu_{z,t}.$$ 
Recall that $r_0(\epsilon)=\delta(\epsilon)$ and $n\ge n_3(\epsilon)\ge n_1(\epsilon)$, where $\delta(\epsilon)$ and $n_1(\epsilon)$ are the constant and the integer appearing in the property \eqref{eq: property nu-infty non-concentration and entropy 1}.
Thus by \eqref{eq: property nu-infty non-concentration and entropy 1},  the set 
$$A(\nu,z,t):=\left\{  k\in \N : \sup_{y\in K}\nu^{\mathcal{A}_k^t(z)}(B(y,\delta(\epsilon)))\le \epsilon \  \textrm{ and } H(\nu^{\mathcal{A}_k^t(z)},\mathcal{D}_n)  \ge n(\gamma\log 2-2\epsilon)\right\}$$
has lower density  at least $1-2\epsilon$ in $\N$.
On the other hand, by the above property (i), the density of $\bigcup_{i=1}^{N(\widetilde{n},\epsilon)}H(C_i,z,t)$ in $\N$ is at least $1-\epsilon$. Note also that the $H(C_i,z,t)$'s are disjoint. It follows that there exists at least one $1\le i_0\le N(\widetilde{n},\epsilon)$ such that the lower density of $A(\nu,z,t)\cap H(C_{i_0},z,t) $ in $ H(C_{i_0},z,t)$ is at least $1-3\epsilon$. Let $$
D=C_{i_0}\ \textrm{ and } \ \mathcal{N}=A(\nu,z,t)\cap H(C_{i_0},z,t). 
$$
Since $H(C_{i_0},z,t)$ has density at least $(1-\epsilon)$ in $H(\widetilde{C}_{i_0},z,t)$, we deduce that the lower density of $\mathcal{N}$ in $H(\widetilde{C}_{i_0},z,t) $ is at least $(1-3\epsilon)(1-\epsilon)$. 
Now, since $(z,t)\in A'$, the sequence
$$\left\{R_\theta^k(t)\in [0,1): k\in H(\widetilde{C}_{i_0},z,t)\right\}$$ 
is UD in $[0,1)$. From Lemma \ref{lemma positive density positive measure 1}, we obtain
$$\mathcal{L}\left(\overline{\{R^k_\theta(t): k\in \mathcal{N}\}}\right)\ge  (1-3\epsilon)(1-\epsilon)\ge 1-4\epsilon.$$
Let us now consider the projection $\Pi_1(D)$. By the above property (ii), we have $$\sharp\left\{E\in \mathcal{B}_{\widetilde{n}}: D\cap E\neq \emptyset\right\}\le e^{\epsilon \widetilde{n}}.$$ It follows that 
$$\sharp\left\{A\in \Pi_1(\mathcal{B}_{\widetilde{n}}): \Pi_1(D)\cap A\neq \emptyset\right\}\le e^{\epsilon \widetilde{n}}.$$ Recall that each element of $\Pi_1(\mathcal{B}_{\widetilde{n}})$ is in  $\mathcal{A}^t_{\widetilde{n}}$ for some $t\in [0,1)$. By definition, it is clear that each element in $\mathcal{A}^t_{\widetilde{n}}$ can be covered by $C_2$ balls of diameter $\alpha^{\widetilde{n}}$, where $C_2$ is a constant depending only on the geometry of $\R^2$, $\alpha$ and $\beta$.
By the choice of $\widetilde{n}$, we have $\alpha^{\widetilde{n}}\le 2^{-n}$. Thus we get 
$$n^{-1}\log N_{2^{-n}}(\Pi_1(D))\le C_3\epsilon$$
for some constant $C_3$ depending only on $\R^2$, $\alpha$ and $\beta$.
It remains to show the property (5) of Proposition \ref{proposition good geometric picture 1}. For this,  we use the fact that for each $k\in \mathcal{N}$, the measure $\nu^{\mathcal{A}_k^t(z)}$ satisfies the  inequalities in the definition of $A(\nu,z,t)$ and  apply Lemma \ref{lemma non-concent and large ent consequence 1} to $\nu^{\mathcal{A}_k^t(z)}$ to get
$$\inf_{y\in K} \frac{1}{n\log 2}H\left(\nu^{\mathcal{A}_k^t(z)}|_{B(y,r_0(\epsilon))^{c}},\mathcal{D}_n\right)\ge \gamma-C_4\epsilon^{\frac{1}{2}} $$
for some constant $C_4$ depending only on $\R^2$, $\alpha$ and $\beta$.
Note that to effectively apply Lemma \ref{lemma non-concent and large ent consequence 1} we need to assume that $n\ge n_3(\epsilon)$ was chosen large enough so that $2^{-n}\le r_0 (\epsilon)$. For this we may replace $n_3(\epsilon)$, if necessary, by a larger number which we contunue to 
denote by $n_3(\epsilon)$, such that $2^{-n_3(\epsilon)}\le r_0 (\epsilon)$.
Letting $C=\max\{C_3,4,C_4^2\}$ we get that the chosen $\nu,D$ and $\mathcal{N}$ satisfy the properties (1)-(5) of Proposition \ref{proposition good geometric picture 1} provided that in (1)-(5) we replace $\epsilon$ by $C\epsilon$. To complete the proof, we only need to replace $r_0(\epsilon)$ and $n_3(\epsilon)$ by $r_0(\epsilon/C)$ and $n_3(\epsilon/C)$, respectively. 
\end{proof}


\subsection{Proof of Theorem \ref{main theorem 1}} \label{section Proof of main Theorem 1}
Recall that we initially assumed  \eqref{eq: assumption slice dim} and we need to prove $\dim_{\rm H}K\ge 1+\gamma$. 
Since $K=C_\alpha\times C_\beta$ and $\dim_{\rm H}C_\alpha=\overline{\dim}_{\rm B}C_\alpha$ and $\dim_{\rm H}C_\beta=\overline{\dim}_{\rm B}C_\beta$, by Lemma \ref{lemma Dimensions of product sets}, $\dim_{\rm H}K=\overline{\dim}_{\rm B}K$.
Thus it suffices to show that $\overline{\dim}_{\rm B}K\ge 1+\gamma$.

Fix a small $\epsilon>0$. Let $r_0=r_0(\epsilon)$ and $n_3=n_3(\epsilon)$ be as in Proposition \ref{proposition good geometric picture 1}. Fix any large $n\ge n_3$. Choose a point $(z,t)\in K\times [0,1)$, a measure $\nu\in \mathcal{P}(K)$, a measurable set $D\subset K\times [0,1)$ and a subset $\mathcal{N}\subset \N$ satisfying the properties (1)--(5) of Proposition \ref{proposition good geometric picture 1}. 

We claim that for any $k\in \mathcal{N}$,
\begin{equation}\label{eq: proof theorem main 1}
\inf_{y\in K} \frac{1}{n\log 2}\log N_{2^{-n}}\left({\rm supp}\left(\nu^{\mathcal{A}_k^t(z)}\right)\setminus B(y,r_0)\right)\ge \gamma-o(1) \ \textrm{ as $\epsilon\to 0$ and $n\to \infty$}.
\end{equation}
The claim is just a consequence of  the property (5) and the elementary formula \eqref{eq: lemma non-concent and large ent consequence, basic 1}.

Note that since $\nu\in \mathcal{P}(\ell\cap K)$ for some line $\ell$ with slope $\beta^{-t}$, $\nu^{\mathcal{A}_k^t(z)}$ is a measure supported on some other slice $\ell'\cap K$ with slope $\beta^{-R_\theta^k(t)}$. Note also that for each $k\in\mathcal{N}$, we have $\Pi_1(U^k(z,t))\in \Pi_1(D)$ and the support of $\nu^{\mathcal{A}_k^t(z)}$ intersects $\Pi_1(D)$.

Let us  summarize the consequences of the properties (1)--(5): For any $\epsilon>0$, there exist a set $F=\{R^k_\theta(t): k\in \mathcal{N}\}\subset [0,1)$ with $\mathcal{L}\left(\overline{F}\right)\ge 1-C\epsilon$ and a set $D_1=\Pi_1(D)\subset K$ with $n^{-1}\log N_{2^{-n}}(D_1)\le C\epsilon$ such that for each $s\in F$ there exists a line $\ell=\ell_s$ with slope $\beta^{-s}$ intersecting $D_1$ and satisfying 
\begin{equation}\label{eq: proof theorem main 3}
\inf_{y\in K} \frac{1}{n\log 2}\log N_{2^{-n}}\left(\ell\cap K\setminus B(y,r_0)\right)\ge \gamma-o(1)\ \textrm{ as $\epsilon\to 0$ and $n\to \infty$}.
\end{equation}
Now, let us consider the set $\widetilde{K}:=K-D_1=\{w-v:w\in K, v\in D_1\}$. It follows from the above  summarized property that for any $t\in F$, we can find some line $\ell=\ell_t'$ with slope $\beta^{-t}$ satisfying  \eqref{eq: proof theorem main 3} and  passing through an $n$-th level dyadic cube containing the origin.
From this, it is easy to check that we have
$$\frac{\log N_{2^{-n}}(\widetilde{K})}{n\log 2}\ge 1+\gamma -o(1) \ \textrm{ as } \epsilon\to 0 \textrm{ and } n\to\infty.$$ 
It is a well known fact that for each $d\ge 1$ there exists a constant $C(d)$ such that $N_{2^{-n}}(A+B)\le C(d)N_{2^{-n}}(A)N_{2^{-n}}(B)$ for any $A,B\subset \R^d$.
Since $n^{-1}\log N_{2^{-n}}(D_1)= o(1)$, it follows that 
$$\frac{\log N_{2^{-n}}(K)}{n\log 2}\ge 1+\gamma -o(1) \ \textrm{ as } \epsilon\to 0 \textrm{ and } n\to\infty.$$ 
This implies  $\overline{\dim}_{\rm B}(K)\ge 1+\gamma$.

\medskip

\section{Proof of Theorem \ref{main theorem 2} }\label{section Proof of Theorem main 2}

For proving  Theorem \ref{main theorem 2}, we follow the same scheme as in the proof of Theorem \ref{main theorem 1}. We only give a sketch of the proof.

Let $X$ be a self-similar set satisfying the conditions of Theorem \ref{main theorem 2}. Suppose that there exists a slice $\ell_0\cap X$ with upper box dimension $\gamma>0$. Our aim is to show that we must have $\dim_{\rm H}X\ge 1+\gamma$. 

{\bf Construction of CP-distributions based on $\ell_0\cap X$.} We  will first construct an ergodic CP-distribution $Q$ with dimension at least $ \gamma$ such that $Q_1$-almost every measure is supported on a slice of $X$. 

We first recall some notations. Let $\mathcal{F}=\{f_i(x)=\lambda O_\xi x+t_i\}_{i=1}^m$ be the IFS generating $X$. Recall that $\lambda\in (0,1), t_i\in \R^2$ and $O_\xi$ is the rotation matrix of  angle $2\pi\xi\in [0,2\pi)$ with $\xi$ irrational.  

Write $\Lambda=\{t_i\}_{i=1}^m$. Consider the symbolic space $\Lambda^\N$ endowed with the metric $d_\lambda$ (recall \eqref{eq: def distance simbolic 1}).
Let $\Pi: \Lambda^\N\to X$ be the projection map defined as 
$$\Pi((x_n)_n)=\sum_{n=1}^\infty \lambda^{n-1}O_\xi^{n-1}x_n.$$
Then $X=\Pi(\Lambda^\N)$. Note that since $\mathcal{F}$ satisfies the strong separation condition, the map $\Pi$ is bi-Lipschitz. Let $M: \mathcal{P}(\Lambda^\N)\times \Lambda^\N$ be the magnification operator defined as 
$$M(\mu,x)=(\mu^{[x_1]},\sigma(x)).$$ 

Recall that for some line $\ell_0$ we have $\overline{\dim}_{\rm B} X\cap \ell_0=\gamma$. Let $A=\Pi^{-1}(X\cap \ell_0)$. Since $\Pi$ is bi-Lipschitz, the upper box dimension of $A$ is also $\gamma$. Thus there exists a sequence $n_k\nearrow \infty$ such that 
$$\lim_{k\to\infty}\frac{N_{\lambda^{n_k}}(A)}{-n_k\log \lambda}=\gamma.$$ 
Similarly as in Subsection \ref{subsection construction of CP-dist 1}, we define a sequence of measures $\{\mu_k\}_k$ on $A$:
$$\mu_k=\frac{1}{N_{\lambda^{n_k}}(A)}\sum_{u\in \Lambda^{n_k}: [u]\cap A\neq \emptyset} \delta_{x_u},$$
where $x_u$ is some point in $[u]\cap A$. Then we set
\begin{equation*}
P_k  =  \frac{1}{N_{\lambda^{n_k}}(A)}\sum_{u\in \Lambda^{n_k}: [u]\cap A\neq \emptyset} \delta_{(\mu_k,x_u)} \ \ \textrm{ and }\
Q_k  =  \frac{1}{n_k}\sum_{i=0}^{n_k-1} M^iP_k.
\end{equation*}
 Let $Q$ be an accumulation point of $\{Q_k\}_k$. Then $Q$ is $M$-invariant and adapted, thus it is a CP-distribution. 
Moreover, it has dimension 
$$H(Q)=\int \frac{1}{\log \alpha}\log \mu[x_1]dQ(\mu,x)=\gamma.$$
One can also show that the measure component of $Q$ is supported on measures which are supported on slices of $X$.
Up to replacing $Q$ by one of its ergodic components with dimension $\ge \gamma$, we may assume that $Q$ is an ergodic CP-distribution with dimension at least $ \gamma$ and that $Q$ is supported on  measures which are supported on slices  of $X$.

{\bf The transformation $W$ on  $X$ and a $W$-invariant measure $\nu$.}
Let $W$ be the inverse map of the IFS $\mathcal{F}$ on $X$, that is, the restriction of $W$  on $f_i(X)$ is $f_i^{-1}$. Then $W$ is expanding and rotating, and it transforms a slice $l\cap X$ into finitely many pieces of slices with the angle of each of the transformed slices being rotated by $-\xi$ comparing to that of the initial slice $l$. 

We use $\mathcal{A}_n$ to denote the partition of $X$ given by 
$$\{\Pi([u]): u\in \Lambda^n\}.$$
For any measure $\eta\in \mathcal{P}(X)$ and $x\in {\rm supp}(\eta)$, we write 
$$\eta^{\mathcal{A}_n(x)}=W^n\left(\frac{\eta|_{\mathcal{A}_n(x)}}{\eta(\mathcal{A}_n(x))}\right).$$
Consider the map  $G:\mathcal{P}(\Lambda^\N)\times \Lambda^\N\to \mathcal{P}(X)$ defined by 
$$G(\mu,x)=\Pi \mu.$$
Then $G$ is continuous. Applying the ergodic theorem to the CP-distribution $Q$, we get for $Q$-a.e. $(\mu,x)$,
$$\frac{1}{N}\sum_{n=0}^{N-1}G(M^n(\mu,x))\to \int GdQ\ \ \textrm{as } N\to\infty.$$
By the definition of $M$, we have $G(M^n(\mu,x))=(\Pi\mu)^{\mathcal{A}_n(x)}$. Thus for $Q$-a.e. $(\mu,x)$,
$$\frac{1}{N}\sum_{n=0}^{N-1}(\Pi\mu)^{\mathcal{A}_n(x)}\to \int \Pi\mu dQ\ \ \textrm{as } N\to\infty.$$
Now, with similar arguments as in the proof of Proposition \ref{proposition invariance 1}, we can prove that the measure $\nu:=\int \Pi\mu dQ$ is actually $W$-invariant. Furthermore, by proceeding analogously as in Subsection \ref{subsection Properties and entropy of the measure}, we can show that $\nu$ satisfy a similar property as \eqref{eq: property slice measure 1}: for any $\epsilon>0$, there exist $\delta=\delta(\epsilon)>0$ and $n_0=n_0(\epsilon)\in\N$ such that for $\nu$-a.e. $z\in X$, there exists $\mu\in \mathcal{P}(\Lambda^\N)$ with  $\Pi\mu\in \mathcal{P}(l\cap X)$ for some line $l$ and 
\begin{equation}
\begin{split}
\liminf_{N\to\infty}\frac{1}{N}\sharp \bigg\{1\le k\le N : \sup_{y\in K}(\Pi\mu)^{\mathcal{A}_k(z)}(B(y,\delta))\le \epsilon \ & \textrm{ and } \\
 H((\Pi\mu)^{\mathcal{A}_k(z)},\mathcal{D}_n)  \ge n(\gamma\log 2-\epsilon)\bigg\}>&1-\epsilon \ \textrm{ for all } n\ge n_0.
\end{split}
\end{equation}
Up to taking an ergodic component, we may also assume that $\nu$ is ergodic. 

{\bf Applying the ergodic theoretical result to the system $(X,W,\nu)$, and conclusion.} 
Now, we apply Theorem \ref{proposition ergodic results} to the system $(X,W,\nu)$ and proceed as in Section \ref{section Proof of Theorem main 1} to finally conclude that $\overline{\dim}_{\rm B}(X)\ge 1+\gamma$. Since $X$ has equal Hausdorff and upper box dimensions, we get $\dim_{\rm H} X\ge 1+\gamma.$

\medskip

\section{Embeddings of self-similar sets and  proofs of the remaining statements }\label{section Proofs of the remaining results} 
In this section, we first present and prove an application of Theorem \ref{main theorem 1} in the study of affine embeddings of self-similar sets, and then we complete the proofs of the remaining statements: Theorem \ref{theorem stronger version of intersection conj} and the claim that Conjecture \ref{orbit conjecture} holds outside a set of Hausdorff dimension zero.

\subsection{Embeddings of self-similar sets}
Let $\Phi=\{\phi_i(x)=\alpha_ix+a_i\}_{i=1}^m$ and $\Psi=\{\psi_i(x)=\beta_jx+b_j\}_{j=1}^l$ be two self-similar IFSs on $\R$. We denote their attractors by $X_\Phi$ and $X_\Psi$, respectively. The problem of affine embeddings of self-similar sets was studied in \cite{FHR}. The following conjecture is a special case of \cite[Conjecture 1.2]{FHR}.
\begin{conjecture}\label{embedding conjecture}
Let $\Phi,\Psi$ be the self-similar IFSs defined above. Assume that $X_\Psi$ is not a singleton and $\Phi$ satisfies the SSC and $\dim_{\rm H}X_\Phi<1$. If there exist real numbers $v,u\neq 0$ such that $uX_\Psi+v\subset X_\Phi$, then for each $1\le j\le l$, there exist rational numbers $r_{i,j}\ge 0$ such that $\beta_j=\prod_{i=1}^m\alpha_i^{r_{i,j}}$.
\end{conjecture}

Some special cases of Conjecture \ref{embedding conjecture} have been proved in \cite{FHR}, and more recently in \cite{Algom2016, FX}. As a corollary of Theorem \ref{main theorem 1}, we show that Conjecture \ref{embedding conjecture} holds under the assumption that $\Phi$ is homogeneous. 
\begin{corollary}\label{corollary embedding result}
Under the assumptions of Conjecture \ref{embedding conjecture}, suppose further that $\Phi$ is homogeneous: there exists $0<\alpha<1$ such that $\alpha_i=\alpha$ for each $1\le i\le m$. Then the conclusion of Conjecture \ref{embedding conjecture} holds, i.e., $\log \beta_j/\log \alpha\in \Q$  for each $1\le j\le l$.
\end{corollary}

\begin{proof}[Proof of Corollary \ref{corollary embedding result}]
We first prove the conclusion under the assumption that $X_{\Psi}$ satisfies the SSC.   
Fix any $j_0\in \{1,\cdots,l\}$, we will show that $\log \beta_{j_0}/\log \alpha\in \Q$.  Choose  any $j\in \{1,\cdots,l\}\setminus \{j_0\}$, let $X_{1}$ be the attractor of  the homogeneous self-similar IFS $\{\psi_{j_0}\circ\psi_{j}, \psi_{j}\circ \psi_{j_0}\}$. Since $X_{\Psi}$ satisfies the SSC,  the same holds for $X_{1}$. Note that $X_{1}\subset X_\Psi$, thus by hypothesis we  have $uX_{1}+v\subset X_\Phi$. We claim that $\log (\beta_{j_0}\beta_j)/\log \alpha\in \Q$. Otherwise, by Theorem \ref{main theorem 1} (and part (2) of Remark \ref{Remark to thm main 1}), we would have 
$$\dim_{\rm H}(uX_{1}+v)\cap X_\Phi\le \max\{0,\dim_{\rm H}X_{1}+\dim_{\rm H}X_\Phi-1\}<\dim_{\rm H}X_{1},$$
which contradicts the fact $(uX_{1}+v)\cap X_\Phi=uX_{1}+v$. Similarly, we can consider the IFS $\{\psi_{j_0}\circ\psi_{j}^2, \psi_{j}^2\circ \psi_{j_0}\}$ and deduce that $\log (\beta_{j_0}\beta_j^2)/\log \alpha\in \Q$. Then we get $\log \beta_{j_0}/\log \alpha\in \Q$.  

Now we consider general  $X_{\Psi}$. Fix any $j_1\in \{1,\cdots,l\}$, we will show that $\log \beta_{j_1}/\log \alpha\in \Q$.   Since $X_{\Psi}$ is not a singleton, there exists $j\in \{1,\cdots,l\}$ such that $\psi_{j_1}$ and $\psi_{j}$ have different fixed points.  From this we deduce that for large enough $n$ the IFS $\{\psi_{j_1}^n, \psi_{j}^n\}$ satisfies the SSC. Let $X_{2}$ be the attractor of this IFS. Then we have $uX_{2}+v\subset X_\Phi$. From this and what we have just proved, we deduce that $\log \beta_{j_1}^n/\log \alpha\in \Q,$ which in turn implies that $\log \beta_{j_1}/\log \alpha\in \Q$.
\end{proof}

\subsection{Proofs of the remaining statements}
We first complete the proof of  Theorem \ref{theorem stronger version of intersection conj}.
Following Furstenberg, we call $C\subset \R$ a $p$-Cantor set if it is the attractor of certain IFS $\mathcal{F}=\{x/p+i/p\}_{i\in \Lambda}$ for some $\Lambda\subset \{0,\cdots, p-1\}$. Clearly, each $p$-Cantor set is a regular $1/p$-self-similar set.

\begin{proposition}\label{prop approx invariant sets by self-similar sets}
Let $A\subset \T=[0,1)$ be a $T_m$-invariant closed set. Then for any $\epsilon >0$, there exist $k\in \N$ and an $m^k$-Cantor set $\widetilde{A}$ such that $A\subset \widetilde{A}$ and $\dim_{\rm H}A\ge \dim_{\rm H}\widetilde{A}-\epsilon$.
\end{proposition}

\begin{proof}
Let us denote by $\mathcal{D}^m_k$ the set of $k$-th level $m$-adic intervals of $\T=[0,1)$, i.e., $\mathcal{D}^m_k=\left\{[i/m^k,(i+1)/m^k): 0\le i\le m^k-1\right\}$. Let $N_{m^{-k}}(A)$ be the number of elements in  $\mathcal{D}^m_k$ intersecting $A$. It is a classical result, due to Furstenberg \cite{Furstenberg67}, that any $T_m$-invariant closed set has equal Haudorff and box dimensions. Thus we have
$$\dim_{\rm H}A=\lim_{k\to\infty}\frac{\log N_{m^{-k}}(A)}{k\log m}.$$
Let us fix a large enough $k$ such that $\frac{\log N_{m^{-k}}(A)}{k\log m}\le \dim_{\rm H} A+\epsilon$. We consider the IFS 
$$\mathcal{F}=\left\{\frac{1}{m^k}x+\frac{i}{m^k}: 0\le i\le m^k-1 \ \textrm{  and } [i/m^k,(i+1)/m^k)\cap A\neq \emptyset\right\}.$$
Since $A$ is $T_m$-invariant, it is also $T_m^k$-invariant, from which we deduce that $A$ is a sub-attractor of $\mathcal{F}$, i.e., $A\subset \bigcup_{f\in \mathcal{F}}f(A)$.
 Let $\widetilde{A}$ be the attractor of $\mathcal{F}$. Then $\widetilde{A}$ is a $m^k$-Cantor set and $A\subset \widetilde{A}$. Now, it remains to show $\dim_{\rm H}A\ge \dim_{\rm H}\widetilde{A}-\epsilon$. For this, we only need to notice that $\widetilde{A}$ satisfies the open set condition and it is well known that
 its Hausdorff dimension is $\frac{\log N_{m^{-k}}(A)}{k\log m}$. By the choice of $k$, we get the desired result.

\end{proof}

\begin{proof}[Proof of Theorem \ref{theorem stronger version of intersection conj}]
Let $A\subset \T$ be closed and $T_p$-invariant and let $B\subset \T$ be closed and $T_q$-invariant, with $p\nsim q$. Fix any $\epsilon>0$. By Proposition \ref{prop approx invariant sets by self-similar sets}, for some large $k$ and $l$, there exist a $p^k$-Cantor set $\widetilde{A}$ and a $q^l$-Cantor set $\widetilde{B}$ such that $A\subset \widetilde{A}$, $\dim_{\rm H}A\ge \dim_{\rm H}\widetilde{A}-\epsilon$, $B\subset \widetilde{B}$ and $\dim_{\rm H}B\ge \dim_{\rm H}\widetilde{B}-\epsilon$. Now, from the hypothesis $p\nsim q$ we deduce that $p^k\nsim q^l$, thus we can apply Theorem \ref{main theorem 1} to the sets $\widetilde{A}$ and $\widetilde{B}$ to get
$$  \overline{\dim}_{\rm B}(u\widetilde{A}+v)\cap \widetilde{B}\le \max \{0,\dim_{\rm H}\widetilde{A}+\dim_{\rm H}\widetilde{B}-1\}.$$
From this we deduce that
$$\overline{\dim}_{\rm B}(uA+v)\cap B \le \max \{0,\dim_{\rm H}A+\dim_{\rm H}B-1\}+2\epsilon$$
Since $\epsilon $ is arbitrary, we get the desired result.
\end{proof}

We now show that Conjecture \ref{orbit conjecture} holds outside a set of Hausdorff dimension zero.
\begin{theorem}\label{thm zero dim exception set for orbit conjecture}
If  $p\nsim q$, then the set of $x\in [0,1]$ which do not satisfy
$$\dim_{\rm H}\overline{O_p(x)}+\dim_{\rm H}\overline{O_q(x)}\ge 1$$
has Hausdorff dimension zero; in fact it is a countable union of sets with upper box dimension zero.
\end{theorem}
\begin{proof}
Let $E=\left\{x\in [0,1]:\dim_{\rm H}\overline{O_p(x)}+\dim_{\rm H}\overline{O_q(x)}< 1\right\}$. We need to show that the set $E$ is a countable union of sets with upper box dimension zero. 

In the following, by a $T_m$-invariant set we always mean a $T_m$-invariant and closed set of $[0,1]$.
Let $$F_1=\left\{(A,B):A  \textrm{ is a } T_p\textrm{-invariant set}, B  \textrm{ is a } T_q\textrm{-invariant set} \textrm{ and } \dim_{\rm H}A+\dim_{\rm H}B< 1\right\}$$
and
$$F_2=\left\{(\widetilde{A},\widetilde{B}):\widetilde{A}  \textrm{ is a } p^k\textrm{-Cantor set}, \widetilde{B}  \textrm{ is a } q^l\textrm{-Cantor set} \textrm{ and } \dim_{\rm H}\widetilde{A}+\dim_{\rm H}\widetilde{B}< 1, k,l\in \N\right\}.$$
By Proposition \ref{prop approx invariant sets by self-similar sets}, for each pair $(A,B)\in F_1$, there exists $(\widetilde{A},\widetilde{B})\in F_2$ such that $A\subset \widetilde{A}$ and $B\subset \widetilde{B}$.
Thus we have
$$E\subset \bigcup_{(A,B)\in F_1}A\cap B\subset \bigcup_{(\widetilde{A},\widetilde{B})\in F_2}\widetilde{A}\cap \widetilde{B}.$$
Now, note that for each $k\in \N$ there are only finitely many $p^k$-Cantor sets and finitely many $q^k$-Cantor sets. Thus the cardinality of $F_2$ is at most countable. Since $p\nsim q$, we have $p^k\nsim q^l$ for any $k,l\in \N$. Thus by Theorem \ref{main theorem 1}, for each $(\widetilde{A},\widetilde{B})\in F_2$, we have
$$\overline{\dim}_{\rm B}(\widetilde{A}\cap \widetilde{B})\le \max\{0,\dim_{\rm H}\widetilde{A}+\dim_{\rm H}\widetilde{B}-1\}=0.$$
Hence $E$ is contained in a countable union of sets with upper box dimension zero. 
\end{proof}

\end{document}